\documentclass[twoside,11pt,leqno]{article}
\usepackage{amsthm,amsmath,amssymb,epsf}
\usepackage[english]{babel}

\newcommand{\R}{{\mathbb R}}
\newcommand{\Q}{{\mathbb Q}}
\newcommand{\N}{{\mathbb N}}

\newcommand{\EE}{{\mathbb E}}
\newcommand{\PP}{{\mathbb P}}
\newcommand{\cA}{{\mathcal A}}
\newcommand{\cX}{{\mathcal X}}

\newtheorem{defn}{Definition}
\newtheorem{res}{Lemma}

\newtheorem{thm}{Theorem}

\renewenvironment{proof}{\noindent{\bf Proof:} }{\hfill $\square$ \\}

\begin{document}
\bibliographystyle{plain}
\thispagestyle{empty}
\begin{center}
{\Large\sc Non-parametric indices of dependence between components for
inhomogeneous multivariate random measures and marked sets}\\[.5in] 

\noindent
{\large M.N.M. van Lieshout}\\[.2in]
\noindent
{\em CWI, P.O.~Box 94079, NL-1090 GB  Amsterdam\\
University of Twente, P.O.~Box 217, NL-7500 AE Enschede\\
The Netherlands}\\[.2in]
In memory of J.\ Oosterhoff.
\end{center}

\begin{verse}
{\footnotesize
\noindent
{\bf Abstract}\\
\noindent
We propose new summary statistics to quantify the association between
the components in coverage-reweighted moment stationary multivariate 
random sets and measures. They are defined in terms of the 
coverage-reweighted cumulant densities and extend classic functional 
statistics for stationary random closed sets. We study the relations
between these statistics and evaluate them explicitly for a range of 
models. Unbiased estimators are given for all statistics and applied 
to simulated examples.
\\[0.2in]

\noindent
{\em Keywords \& Phrases:}
compound random measure, coverage measure, coverage-reweighted moment 
stationarity, cross hitting functional, empty space function, germ-grain
model, $J$-function, $K$-function, moment measure, multivariate random 
measure, random field model, reduced cross correlation measure, spherical 
contact distribution.

}
\end{verse}

\section{Introduction}

Popular statistics for investigating the dependencies between different 
types of points in a multivariate point process include cross-versions of 
the $K$-function \cite{Ripl88}, the nearest-neighbour distance distribution 
\cite{Digg83} or the $J$-function \cite{LiesBadd96}. Although originally 
proposed under the assumption that the underlying point process distribution 
is invariant under translations, in recent years all statistics mentioned 
have been adapted to an inhomogeneous context. More specifically, 
for univariate point processes, \cite{Badd00} proposed an inhomogeneous 
extension of the $K$-function, whilst \cite{Lies11} did so for the 
nearest-neighbour distance distribution and the $J$-function. An 
inhomogeneous cross $K$-function was proposed in \cite{MollWaag04}, 
cross nearest-neighbour distance distributions and $J$-functions were
introduced in \cite{Lies11} and further studied in \cite{CronLies16}.

Although point processes can be seen as the special class of random 
measures that take integer values, functional summary statistics for
random measures in general do not seem to be well studied. An exception 
is the pioneering paper by Stoyan and Ohser \cite{StoyOhse82} in which,
under the assumption of stationarity, two types of characteristics were
proposed for describing the correlations between the components of a 
multivariate random closed sets in terms of their coverage measures. The 
first one is based on the second order moment measure \cite{DaleVere08}
of the coverage measure \cite{Molc05}, the second one on the capacity 
functional \cite{Math75}. The authors did not pursue any relations 
between their statistics. Our goal in this paper is, in the context
of multivariate random measures, to define generalisations of 
the statistics of \cite{StoyOhse82} that allow for inhomogeneity, 
and to investigate the relations between them. 

The paper is organised as follows. In Section~\ref{S:moments}
we review the theory of multivariate random measures. We recall the 
definition of the Laplace functional and Palm distribution and discuss 
the moment problem. We then present the notion of coverage-reweighted 
moment stationarity. In Section~\ref{S:statistics} we introduce new 
inhomogeneous counterparts to Stoyan and Ohser's reduced cross
correlation measure. In the univariate case, the latter coincides 
with that proposed by Gallego {\em et al.\/} for germ-grain models 
\cite{Gall14}. We go on to propose a cross $J$-statistic and relate
it to the cross hitting intensity \cite{StoyOhse82} and empty space
function \cite{Math75} defined for stationary random closed sets. Next, 
we give explicit expressions for our functional statistics for a range 
of bivariate models: compound random measures including linked and 
balanced models, the coverage measure associated to random closed
sets such as germ-grain models, and random field models with 
particular attention to log-Gaussian and thinning random fields.
Then, in Section~\ref{S:estimation}, we turn to estimators for 
the new statistics and apply them to simulations of the models 
discussed in Section~\ref{S:theo-ex}.

\section{Random measures and their moments}
\label{S:moments}

In this section, we recall the definition of a multivariate random
measure \cite{Chiu13,DaleVere08}.

\begin{defn}
Let $\cX = \R^d \times \{ 1, \dots, n \}$, $d, n\in\N$, be equipped
with the metric $d(\cdot, \cdot)$ defined by 
$d( (x,i), (y,j) ) = || x - y || + | i - j |$
for $x, y \in \R^d$ and $i, j \in \{ 1, \dots, n \}$. Then a 
multivariate random measure $\Psi$ on $\cX$ is a measurable mapping
from a probability space $(\Omega, \cA, \PP)$ into the space of
all locally finite Borel measures on $\cX$ equipped with the smallest 
$\sigma$-algebra that makes all $\Psi_i(B)$ with $B\subset \R^d$ ranging
through the bounded Borel sets and $i$ through $\{ 1, \dots, n \}$
a random variable. 
\end{defn}

An important functional associated with a multivariate random measure 
is its {\em Laplace functional\/}.

\begin{defn}
Let $\Psi = (\Psi_1, \dots, \Psi_n)$ be a multivariate random measure.
Let $u: \R^d \times \{ 1, \ldots, n \} \to \R^+$ be a bounded 
non-negative measurable function such that the projections 
$u(\cdot, i): \R^d \to \R^+$, $i=1, \dots, n$, have bounded support. Then
\[
L(u) = \EE \exp\left[ - \sum_{i=1}^n \int_{\R^d}  u(x,i) \, d\Psi_i(x) \right]
\]
is the Laplace functional of $\Psi$ evaluated at $u$.
\end{defn}

The Laplace functional completely determines the distribution of the 
random measure $\Psi$ \cite[Section~9.4]{DaleVere08} 
and is closely related to the {\em moment measures\/}. 
First, consider the case $k=1$. Then, for Borel sets $B \subseteq \R^d$
and $i\in \{ 1, \dots, n \}$, set
\[
\mu^{(1)}( (B \times \{ i \}) = \EE \Psi_i(B).
\]
Provided the set function $\mu^{(1)}$ is finite for bounded Borel sets, 
it yields a locally finite Borel measure that is also denoted by
$\mu^{(1)}$ and referred to as the {\em first order moment measure\/} of
$\Psi$. More generally, for $k\geq 2$, the $k$-th order moment measure 
is defined by the set function
\[
\mu^{(k)}( (B_1 \times \{ i_1 \}) \times \cdots \times (B_k \times \{ i_k \}) )
 = 
\EE\left( 
 \Psi_{i_1}(B_1) \times \cdots \times \Psi_{i_k}( B_k) 
\right)
\]
where $B_1, \dots, B_k\subseteq \R^d$ are Borel sets and $i_1, \dots, i_k
\in \{ 1, \dots, n \}$. If $\mu^{(k)}$ is finite for bounded $B_i$, it
can be extended uniquely to a locally finite Borel measure on $\cX^k$, cf.\
\cite[Section~9.5]{DaleVere08}.

In the sequel we shall need the following relation between the Laplace
functional and the moment measures. Let $u$ be a bounded non-negative
measurable function $u:\R^d \times \{ 1, \dots, n \} \to \R^+$ such
that its projections have bounded support. Then,
\begin{equation}
L(u)  =  1 + \sum_{k=1}^\infty \frac{(-1)^k}{k!} 
  \sum_{i_1 = 1}^n \int_{\R^d}
\cdots \sum_{i_k=1}^n \int_{\R^d}
u(x_1, i_1) \cdots u(x_k, i_k) \, d\mu^{(k)}((x_1, i_1), \dots, (x_k, i_k) ) 
\label{e:lu}
\end{equation}
provided that the moment measures of all orders exist and that the series on 
the right is absolutely convergent \cite[(6.1.9)]{DaleVere03}.

The above discussion might lead us to expect that the moment measures 
determine the distribution of a random measure. As for a random variable, 
such a claim cannot be made in complete generality. However, Zessin 
\cite{Zess83} derived a sufficient condition.

\begin{thm}
\label{t:zessin}
Let $\Psi = (\Psi_1, \dots, \Psi_n)$ be a multivariate random measure
and assume that the series
\[
\sum_{k=1}^\infty \mu^{(k)}((B\times C)^k)^{-1/(2k)} = \infty
\]
diverges for all bounded Borel sets $B \subset \R^d$ and all 
$C \subseteq \{ 1, \dots, n \}$.
Then the distribution of $\Psi$ is uniquely determined by its moment
measures.
\end{thm}

The existence of the first-order moment measure implies that of a
Palm distribution \cite[Prop.~13.1.IV]{DaleVere08}.

\begin{defn}
\label{d:Palm}
Let $\Psi = (\Psi_1, \dots, \Psi_n)$ be a multivariate random measure
for which $\mu^{(1)}$ exists as a locally finite measure.
Then $\Psi$ admits a Palm distribution $P^{(x,i)}$ which is defined uniquely 
up to a $\mu^{(1)}$-null-set and satisfies
\begin{equation}
\label{e:CM}
\EE\left[ 
\sum_{i=1}^n \int_{\R^d} g((x, i), \Psi) \, d\Psi_i(x) 
\right]
=
\sum_{i=1}^n \int_{\R^d} \EE^{(x,i)} \left[ g((x, i), \Psi) \right]
\, d\mu^{(1)}(x,i)
\end{equation}
for any non-negative measurable function $g$. Here, $\EE^{(x,i)}$ denotes
expectation with respect to $P^{(x,i)}$.
\end{defn}

The equation (\ref{e:CM}) is sometimes referred to as the Campbell--Mecke
formula.

Next, we will focus on random measures whose moment measures are
absolutely continuous. Thus, suppose that
\[
\mu^{(k)}( (B_1 \times \{ i_1 \}) \times \cdots \times (B_k \times \{ i_k \}) )
 = 
\int_{B_1} \cdots \int_{B_k} p_k( (x_1, i_1), \dots, (x_k, i_k) ) \,
dx_1 \cdots dx_k,
\]
or, in other words, that $\mu^{(k)}$ is absolutely continuous with
Radon--Nikodym derivative $p_k$, the $k$-point {\em coverage function\/}.
The family of $p_k$s define cumulant densities as follows \cite{DaleVere08}.

\begin{defn}
Let $\Psi = (\Psi_1, \dots, \Psi_n)$ be a multivariate random measure and
assume that its moment measures exist and are absolutely continuous.
Assume that the coverage function $p_1$ is strictly positive. Then the 
coverage-reweighted cumulant densities $\xi_k$ are defined recursively 
by $\xi_1 \equiv 1$ and, for $k \geq 2$,
\[
\frac{p_{k}((x_1, i_1), \dots, (x_k, i_k))}{p_1(x_1, i_1) \cdots p_1(x_k, i_k)} =
\sum_{m=1}^k \sum_{D_1, \dots, D_m} 
\prod_{j=1}^m \xi_{|D_j|} ( \{ (x_l, i_l) : l \in D_j \} )
\]
where the sum is over all possible partitions $ \{ D_1, \dots, D_m \}$,
$D_j \neq \emptyset$, of $\{ 1, \dots, k \}$.
Here we use the labels $i_1, \dots, i_k$ to define which of the 
components is considered and denote the cardinality of $D_j$ by
$|D_j|$.
\end{defn}

For the special case $k=2$,
\[
\xi_2( (x_1, i_1), (x_2, i_2) ) 
= \frac{ p_2( (x_1, i_1), (x_2, i_2) ) - p_1( x_1, i_1) \, p_1(x_2, i_2)}{
p_1( x_1, i_1) \, p_1(x_2, i_2) }.
\]
Consequently, $\xi_2$ can be interpreted as a coverage-reweighted 
covariance function. 

\begin{defn}
Let $\Psi = (\Psi_1, \dots, \Psi_n)$ be a multivariate 
random measure. Then $\Psi$ is coverage-reweighted moment 
stationary if its coverage function exists and is bounded away
from zero, $\inf p_1( x, i) > 0$, and its
coverage-reweighted cumulant densities $\xi_k$ exist and are 
translation invariant in the sense that
\[
\xi_k( (x_1 + a, i_1), \dots, (x_k + a, i_k) ) = 
\xi_k( ( x_1, i_1), \dots, (x_k, i_k) )
\]
for all $a\in \R^d$, $i_j\in \{1, \dots, n\}$ and almost all $x_j\in \R^d$.
\end{defn}

An application of \cite[Lemma~5.2.VI]{DaleVere03} to (\ref{e:lu})
implies that 
\begin{eqnarray}
\nonumber
\log L(u) & = &  \\
& & \sum_{k=1}^\infty 
\frac{(-1)^k}{k!} 
  \sum_{i_1=1}^n \int_{\R^d} \cdots \sum_{i_k=1}^n \int_{\R^d}
\xi_k( (x_1, i_1), \dots, (x_k, i_k) ) 
\prod_{j=1}^k u(x_j, i_j) \, p_1(x_j,i_j) \, dx_j
\label{e:loglu}
\end{eqnarray}
provided the series is absolutely convergent.

The next result states that the Palm moment measures of the
coverage-reweighted random measure can be expressed
in terms of those of $\Psi$.

\begin{thm}
\label{t:palm}
Let $\Psi$ be a coverage-reweighted moment stationary multivariate
random measure and $k\in\N$.
Then for all bounded Borel sets $B_1, \dots, B_k$ and all $i_1, \dots,
i_k \in \{ 1, \dots, n \}$, the Palm expectation
\[
 \EE^{(a,i)}\left[ 
\int_{a+B_1} \cdots \int_{a+B_k}
\frac{d\Psi_{i_1}(x_1) \cdots d\Psi_{i_k}(x_k) }{
 p_1( x_1, i_1) \cdots p_1( x_k, i_k ) } 
\right] =
\]
\[
\quad \quad \quad
= \int_{B_1} \cdots \int_{B_k} 
\frac{ p_{k+1}( (0,i), (x_1, i_1), \dots, (x_k, i_k) ) }{ p_1( 0,i )\,
p_1(x_1, i_1) \cdots p_1(x_k, i_k) } \, dx_1 \cdots dx_k
\]
for almost all $a\in \R^d$.  
\end{thm}

\begin{proof}
By (\ref{e:CM}) with $g( (a,j), \Psi) = 0$ if $j\neq i$ and
\[
g( (a,i), \Psi) =  \frac{ 1_A(a) }{p_1(a,i)}  \int_{a+B_1} \cdots \int_{a+B_k}
\frac{1}{ p_1( x_1, i_1 ) \cdots p_1( x_k, i_k ) }\, d\Psi_{i_1}(x_1)
\cdots d\Psi_{i_k}(x_k),
\]
for some bounded Borel sets $A, B_1, \dots, 
B_k \subset \R^d$ and any $i, i_1, \dots, i_k \in \{ 1, \dots, n \}$,
one sees that
\[
\EE\left[ \int_A \frac{1}{p_1(a,i)}
\int_{a+B_1} \cdots \int_{a+B_k}
\frac{1}{ p_1( x_1, i_1 ) \cdots p_1( x_k, i_k ) } \, d\Psi_{i_1}(x_1)
\cdots d\Psi_{i_k}(x_k) \,  d\Psi_i(a)
\right] =
\]
\[
= \int_A \EE^{(a,i)}\left[ 
\int_{a+B_1} \cdots \int_{a+B_k}
\frac{1}{ p_1(a,i)\, p_1( x_1, i_1 ) \cdots p_1( x_k, i_k ) } 
d\Psi_{i_1}(x_1) \cdots d\Psi_{i_k}(x_k) 
\right] p_1(a,i) \, da .
\]
The left hand side is equal to
\[
 \int_A
 \left[ \int_{B_1} \cdots \int_{B_k} \frac{p_{k+1}( (a,i), (a+x_1, i_1), 
\dots , (a+x_k, i_k) ) }{ p_1(a,i) \,
p_1(a+x_1, i_1) \cdots p_1(a+x_k, i_k) 
} \, dx_1 \cdots dx_k \right] da
\]
and the inner integrand does not depend on the choice of $a\in A$
by the assumptions on $\Psi$.
Hence, for all bounded Borel sets $A\subset \R^d$,
\[
\int_A \EE^{(a,i)}\left[ 
\int_{a+B_1} \cdots \int_{a+B_k}
\frac{d\Psi_{i_1}(x_1) \cdots d\Psi_{i_k}(x_k) }{
 p_1( x_1, i_1) \cdots p_1( x_k, i_k ) } 
\right]
 da =
\]
\[
=
 \int_A 
 \left[ \int_{B_1} \cdots \int_{B_k} \frac{p_{k+1}( (0,i), (x_1, i_1), 
\dots , (x_k, i_k) ) }{ p_1(0,i) \,
p_1(x_1, i_1) \cdots p_1(x_k, i_k) 
} \, dx_1 \cdots dx_k \right] da.
\]
Therefore the Palm expectation takes the same value for almost all
$a\in\R^d$ as claimed. 
\end{proof}

\section{Summary statistics for multivariate random 
measures}
\label{S:statistics}

\subsection{The inhomogeneous cross $K$-function}

For the coverage measures associated to a stationary bivariate 
random closed set, Stoyan and Ohser \cite{StoyOhse82} defined
the {\em reduced cross correlation measure\/} as follows.
Let $B(x,t)$ be the closed ball of radius $t\geq 0$ centred
at $x\in\R^d$ and set, for any bounded Borel set $B$ of positive
volume $\ell(B)$,
\begin{equation}
\label{e:K-stat}
R_{12}(t) = \frac{1}{p_1( 0,1) \, p_1(0,2) }
\EE\left[
\frac{1}{\ell(B)} \int_B \Psi_2( B(x,t) ) \, d\Psi_1(x)
\right].
\end{equation}
Due to the assumed stationarity, the definition does not depend 
on the choice of $B$. In the 
univariate case, Ayala and Sim\'o \cite{AyalSimo98} called a 
function of this type the $K$-function in analogy to a similar 
statistic for point processes \cite{Digg83,Ripl77}. 

In order to modify (\ref{e:K-stat}) so that it applies to more 
general, not necessarily stationary, random measures, we focus 
on the second order coverage-reweighted cumulant density $\xi_2$ 
and assume it is invariant under translations. If additionally 
$p_1$ is bounded away from zero, $\Psi$ is {\em second order 
coverage-reweighted stationary\/}.

\begin{defn} 
\label{d:k12}
Let $\Psi = (\Psi_1, \Psi_2)$ be a bivariate random measure which admits
a second order coverage-reweighted cumulant density $\xi_2$ that is 
invariant under translations and a coverage function $p_1$ that is bounded 
away from zero. Then, for $t\geq 0$, the cross $K$-function is defined by
\[
K_{12}(t) =  \int_{B(0,t)} (1 + \xi_2( (0, 1), ( x, 2) ) ) \, dx .
\]
\end{defn}

Note that the cross $K$-function is symmetric in the components of $\Psi$, 
that is, $K_{12} = K_{21}$. The next result gives an alternative expression
in terms of the expected content of a ball under the Palm distribution 
of the coverage-reweighted random measure.

\begin{res}
\label{l:Kpalm}
Let $\Psi = (\Psi_1, \Psi_2)$ be a second order coverage-reweighted stationary 
bivariate random measure and write $B(a,t)$ for the closed ball of radius 
$t\geq 0$ around $a\in\R^d$. Then
\[
K_{12}(t) = \EE^{(a,1)} \left[ 
\int_{ B(a,t) } \frac{1}{p_1(x,2)} \, d\Psi_2(x) \right]
\]
and the right hand side does not depend on the choice of $a\in\R^d$.
\end{res}

\begin{proof}
Apply Theorem~\ref{t:palm} for $k=1$, $i = 1$, $B_1 = B(0,t)$ and
$i_1=2$ to obtain
\[
\EE^{(a,1)}\left[ \int_{B(a,t)} \frac{1}{p_1(x, 2)} d\Psi_2(x) \right] = 
\int_{B(0,t)} \frac{p_2((0,1),(x,2))}{p_1(0,1) \, p_1(x,2)} \, dx =
\int_{B(0,t)} (1 + \xi_2( (0, 1), ( x, 2) ) ) \, dx .
\]
\end{proof}

To interpret the statistic, recall that $\xi_2$ is equal to the
coverage-reweighted covariance. Thus, if $\Psi_1$ and $\Psi_2$ are 
independent, 
\[
K_{12}(t) = \ell( B(0,t) ),
\]
the Lebesgue measure of $B(0,t)$.
Larger values are due to positive correlation, smaller ones  to
negative correlation between $\Psi_1$ and $\Psi_2$.
Furthermore, if $\Psi = (\Psi_1, \Psi_2)$ is stationary,  Lemma~\ref{l:Kpalm} 
implies that
\[
K_{12}(t) 
 =   \frac{1}{p_1(0,2)} \EE^{(0,1)}\left[ \Psi_2(B(0,t))\right] 
\]
which, by the Campbell--Mecke equation (\ref{e:CM}), is equal to
\[
 \frac{1}{p_1(0,1) p_1(0,2)}
\EE\left[ 
   \frac{1}{\ell(B)} \int_B \Psi_2(B(x,t))  \, d\Psi_1(x)
\right] .
\]
Consequently, $K_{12}(t) = R_{12}(t)$, the reduced cross correlation measure
of \cite{StoyOhse82}.

\subsection{Inhomogeneous cross $J$-function}

The cross $K$-function is based on the second order coverage-reweighted
cumulant density. In this section, we propose a new statistic that 
encorporates the coverage-reweighted cumulant densities of all orders.

\begin{defn}
\label{d:j12}
Let $\Psi = (\Psi_1, \Psi_2)$ be a coverage-reweighted moment stationary
bivariate random measure. For $t\geq 0$ and $k\geq 1$, set
\[
J_{12}^{(k)}(t) = \int_{B(0,t)} \cdots \int_{B(0,t)} 
\xi_{k+1}( (0, 1), (x_1, 2), \dots, (x_k, 2) ) \, dx_1 \cdots dx_k
\]
and define the cross $J$-function by
\[
J_{12}(t) = 1 + \sum_{k=1}^\infty \frac{(-1)^k}{k!} J_{12}^{(k)}(t)
\]
for all $t\geq 0$ for which the series is absolutely convergent.
\end{defn}

Note that 
\[
J_{12}^{(1)}(t) = K_{12}(t) - \ell(B(0,t)).
\]
The appeal of Definition~\ref{d:j12} lies in the fact that its
dependence on the cumulant densities and, furthermore, its 
relation to $K_{12}$ are immediately apparent.
However, being an alternating series, $J_{12}(t)$ is not convenient 
to handle in practice. The next theorem gives a simpler 
characterisation in terms of the Laplace transform.

\begin{thm}
\label{t:j12}
Let $\Psi = (\Psi_1, \Psi_2)$ be a coverage-reweighted moment stationary
bivariate random measure. Then, for $t\geq 0$ and $a\in \R^d$,
\begin{equation}
\label{e:j12}
J_{12}(t) = \frac{L^{(a,1)}(u_t^a)}{L(u_t^a)}
\end{equation}
for $u_t^a(x,i) = 1\{ (x,i) \in B(a,t) \times \{ 2 \} ) / p_1( x, i)$,
provided the series expansions of $L(u_t^a)$ and $J_{12}(t)$ are absolutely 
convergent. In particular, $J_{12}(t)$ does not depend on the choice of 
origin $a\in\R^d$.
\end{thm}

\begin{proof}
First, note that, by (\ref{e:loglu}),
$L(u_t^a)$ does not depend on the choice of $a$. Also,
by Theorem~\ref{t:palm} and the series expansion (\ref{e:lu}) of the Laplace
transform for $u_t^a(x,i)$, provided the series is absolutely convergent,
\begin{eqnarray*}
L^{(a,1)} ( u_t^a ) & = & 1 + \sum_{k=1}^\infty \frac{(-1)^k}{k!} 
  \EE^{(a,1)} \left[ \int_{B(a,t)} \cdots \int_{B(a,t)} 
\frac{d\Psi_2(x_1) \cdots d\Psi_2(x_k)}{p_1(x_1,2) \cdots p_1(x_k,2) }
\right] \\
& = &
 1 + \sum_{k=1}^\infty \frac{(-1)^k}{k!} \int_{B(0,t)} \cdots \int_{B(0,t)}
\frac{ p_{k+1}((0,1), (x_2,2), \dots, (x_{k+1},2)) }{p_1(0,1) p_1(x_2,2) 
\cdots p_1(x_{k+1},2)} \, dx_2 \cdots dx_{k+1} \\
& = &
1 + \sum_{k=1}^{\infty} \frac{(-1)^k}{k!}
\int_{B(0,t)} \cdots \int_{B(0,t)} \sum_{m=1}^{k+1} \sum_{D_1, \ldots, D_m}
\prod_{j=1}^{m} \xi_{|D_j|}(\{(x_l, i_l): l \in D_j\}) 
\prod_{i=2}^{k+1} dx_i,
\end{eqnarray*}
where $(x_1,i_1) \equiv (0,1)$ and $i_l = 2$ for $l>1$.
By splitting the last expression into terms based on whether the sets
$D_j$ contain the index $1$ (i.e.\ on whether $\xi_{|D_j|}$ includes
$(x_1,i_1) \equiv (0,1)$), under the convention that $\sum_{k=1}^{0} = 1$,
we obtain
\[
L^{(a,1)} ( u_t^a ) = 1 +
 \sum_{k=1}^{\infty} \frac{(-1)^k}{k!} \sum_{\Pi\in\mathcal{P}_k}
   J^{(|\Pi|)}_{12}(t) \sum_{m=1}^{k-|\Pi|}
  \sum_{\substack{D_1, \ldots, D_m \neq \emptyset \text{ disjoint}\\
   \cup_{j=1}^{m} D_j = \{1, \ldots, k \} \setminus \Pi}}
     \prod_{j=1}^{m} I_{|D_j|},
\]
where
\[
I_{k} = \int_{B(0,t)} \cdots \int_{B(0,t)}
 \xi_{k}((x_1, 2), \ldots, (x_{k}, 2)) \, dx_1 \cdots dx_k,
\]
$J^{(0)}_{12}(t) \equiv 1$, and $\mathcal{P}_{k}$ is the power set of 
$\{1,\ldots,k\}$. Finally, by noting that the expansion contains terms of
the form $J_{12}^{(k)}(t) I_{k_1}^{m_1} \cdots I_{k_n}^{m_n}$ multiplied by a scalar
and basic combinatorial arguments, we conclude that
\begin{eqnarray*}
L^{(a,1)} ( u_t^a ) & = & 
   \left( 1 +  \sum_{k=1}^{\infty} \frac{(-1)^k}{k!}  J_{12}^{(k)}(t)\right) 
\times 
   \left( 1 + \sum_{k=1}^{\infty} \frac{(-1)^k}{k!} \sum_{m=1}^{k}
   \sum_{\substack{D_1, \ldots, D_m \neq \emptyset\text{ disjoint}\\
     \cup_{j=1}^{m} D_j = \{ 1, \ldots, k \}}} \prod_{j=1}^{m} I_{|D_j|} \right)
\\
& = & J_{12}(t) \, L(u_t^a).
\end{eqnarray*}
The right hand side does not depend on $a$ and is absolutely convergent 
as a product of absolutely convergent terms. Therefore, so is the series 
expansion for $L^{(a,1)}$.
\end{proof}

Heuristically, the cross $J$-function compares expectations 
under the Palm distribution $P^{(0,1)}$ to those under the distribution
$P$ of $\Psi$. If the components of $\Psi$ are independent,
conditioning on the first component placing mass at the origin
does not affect the second component, so $J_{12}(t) = 1$. 
A value larger than $1$ means that such conditioning tends
to lead to a smaller $\Psi_2( B(0,t) )$ content (typical for negative 
association); analogously, $J_{12}(t) < 1$ suggests positive association
between the components of $\Psi$.

\section{Examples}
\label{S:theo-ex}

In this section we calculate the cross $K$- and $J$-statistics for a range 
of well-known models.

\subsection{Compound random measures}

Let 
\(
\Lambda = ( \Lambda_1, \Lambda_2)
\)
be a random vector such that its components take values in 
$\R^+$ and have finite, strictly positive expectation. Set
\begin{equation}
\label{e:Psi-Cox}
\Psi = ( \Lambda_1 \nu, \Lambda_2 \nu )
\end{equation}
for some locally finite Borel measure $\nu$ on $\R^d$ that is 
absolutely continuous with density function $f_\nu \geq \epsilon > 0$. 
In other words, 
\(
\Psi_i(B) = \Lambda_i \int_B f_\nu(x) \, dx = \Lambda_i \, \nu(B).
\)

\begin{thm}
\label{t:compound}
The bivariate random measure (\ref{e:Psi-Cox}) is coverage-reweighted moment 
stationary and 
\begin{eqnarray*}
K_{12}(t) & = & \kappa_d t^d \left( 1 + \frac{ {\rm{Cov}}( \Lambda_1, \Lambda_2) 
 }{\EE (\Lambda_1) \, \EE( \Lambda_2 )} \right) \\
J_{12}(t) & = & 
\frac{\EE\left( \Lambda_1 \exp\left[
   - \Lambda_2 \kappa_d t^d / \EE \Lambda_2  \right] 
\right) }{ \EE\left( \Lambda_1 \right) \EE\left( \exp\left[
   - \Lambda_2 \kappa_d t^d / \EE \Lambda_2  \right] \right)}.
\end{eqnarray*}
\end{thm}

Both statistics do not depend on $f_\nu$. The cross
$J$-function is equal to the $\Lambda_1$-weighted Laplace functional
of $\Lambda_2$ evaluated in $\kappa_d t^d / \EE\Lambda_2$.

To see that both statistics capture a form of `dependence' between the 
components of $\Psi$, note that the cross $K$-function exceeds 
$\kappa_d t^d$ if and only if $\Lambda_1$ and $\Lambda_2$ are positively 
correlated. For the cross $J$-function, recall that two random variables 
$X$ and $Y$ are {\em negatively quadrant dependent\/} if 
Cov$(f(X), g(Y)) \leq 0$ whenever $f, g$ are 
non-decreasing functions, {\em positively quadrant dependent\/} if 
Cov$(f(X), g(Y)) \geq 0$ (provided the moments exist)
\cite{Esar67,Kuma83,Lehm66}. Applied to our context, it follows that 
if $\Lambda_1$ and $\Lambda_2$ are positively quadrant dependent, 
$J_{12}(t) \leq 1$ whilst $J_{12}(t) \geq 1$ if $\Lambda_1$ and 
$\Lambda_2$ are negatively quadrant dependent.

\

\begin{proof}
Since
\[
\EE\left[ \Psi_1(B_1) \cdots \Psi_1(B_k) \Psi_2(B_{k+1}) \cdots \Psi_2(B_{k+l})
\right] = 
\EE\left(  \Lambda_1^k \Lambda_2^l \right) \int_{B_1} \cdots \int_{B_{k+l}}
\prod_{i=1}^{k+l} f_\nu(x_i) \, dx_1 \cdots dx_{k+l},
\]
the coverage function of $\Psi$ is given by
\[
p_{k+l}( (x_1, 1), \dots, (x_k, 1), (x_{k+1}, 2), \dots, (x_{k+l}, 2) )
= 
\EE\left( \Lambda_1^k \Lambda_2^l \right) \prod_{i=1}^{k+l} f_\nu(x_i)
\]
so that the coverage-reweighted cumulant densities of $\Psi$ are 
translation invariant. The assumptions imply that $p_1(x,i) =
\EE(\Lambda_i) f_\nu(x)$ is bounded away from zero. Hence, $\Psi$
is coverage-reweighted moment stationary. 

Specialising to second order, one finds that 
\[
\xi_2( (0,1), (x,2) ) = \frac{\EE( \Lambda_1 \Lambda_2 ) - 
\EE (\Lambda_1) \, \EE( \Lambda_2 )}{\EE (\Lambda_1) \, \EE( \Lambda_2 )} =
\frac{{\rm{Cov}}(\Lambda_1, \Lambda_2)}{\EE (\Lambda_1) \, \EE( \Lambda_2 )}
\]
from which the expression for $K_{12}(t)$ follows upon integration.

As for the cross $J$-function, the denominator in Theorem~\ref{t:j12} 
can be written as 
\[
L(u_t^0)  =  
\EE \exp\left[ 
  - \int_{B(0,t)} \frac{1}{\EE(\Lambda_2) f_\nu(x)} d \Psi_2(x) 
\right] 
 = \EE \exp\left[ - \Lambda_2 \kappa_d t^d / \EE \Lambda_2 \right] .
\]
For the numerator, we need the Palm distribution of $\Lambda_1$.
By \cite[p.~274]{DaleVere08}, $P^{(0,1)}$ is $\Lambda_1$-weighted 
and the proof is complete.
\end{proof}

Let us consider two specific examples discussed in
\cite[Section~6.6]{Digg83}.

\paragraph{Linked model}

Let $\Psi_2 = A \Psi_1$ for some $A>0$. 
Since, for $l_1, l_2 \in\R^+$,
\[
\PP( \Lambda_1 \leq l_1; \Lambda_2 \leq l_2 ) =
\PP( \Lambda_1 \leq \min( l_1, l_2 / A ) ) \geq 
\PP( \Lambda_1 \leq l_1 ) \, \PP( A \Lambda_1  \leq l_2 ),
\]
$\Lambda_1$ and $\Lambda_2$ are positively quadrant dependent
\cite[Theorem~4.4]{Esar67}
and, a fortiori, positively correlated. Therefore
\(
K_{12}(t) \geq \kappa_d t^d
\)
and
\(
J_{12}(t) \leq 1.
\)

\paragraph{Balanced model}

Let $\Lambda_1$ be supported on the interval $(0, A)$ for some $A>0$ and
set $\Lambda_2 =  A - \Lambda_1$. Since, for $l_1, l_2 \in (0, A)$
such that $A - l_2 \leq l_1$,
\[
\PP( \Lambda_1 \leq l_1; \Lambda_2 \leq l_2 ) =
\PP( \Lambda_1 \leq l_1 ) - \PP(\Lambda_1 < A - l_2) \leq
\]
\[
\leq \PP( \Lambda_1 \leq l_1 ) - 
\PP( \Lambda_1 \leq l_1 ) \, \PP(\Lambda_1 < A - l_2) 
 =
\PP( \Lambda_1 \leq l_1 ) \, \PP( \Lambda_2 \leq l_2 ),
\]
$\Lambda_1$ and $\Lambda_2$ are negatively quadrant dependent
\cite{Kuma83}
and, a fortiori, negatively correlated. Therefore
\(
K_{12}(t) \leq \kappa_d t^d
\)
and
\(
J_{12}(t) \geq 1.
\)

\

By Theorem~\ref{t:compound}, the cross $K$-function is increasing
in $t$. It can be shown that under the extra assumption of finite
second order moments, for the linked model, $J_{12}(t)$ is 
monotonically non-increasing. Analogously, in the balanced case, 
$J_{12}(t)$ is non-decreasing \cite{LiesBadd96}. A proof is given 
in the Appendix.

\subsection{Coverage measure of random closed sets}
\label{S:rcs}

Let 
\(
X = (X_1, X_2)
\)
be a bivariate random closed set. Then, by Robbins' theorem 
\cite[Theorem~4.21]{Molc05}, the Lebesgue content
\[
 \ell(X_i\cap B) = \int_B 1\{ x \in X_i \} \, dx
\]
of $X_i \cap B$ is a random variable for every Borel set 
$B \subseteq \R^d$ and every component $X_i$, $i=1, 2$. Letting 
$B$ and $i$ vary, one obtains a bivariate random measure denoted 
by $\Psi$. Clearly, $\Psi$ is locally finite.

Reversely, a bivariate random measure $\Psi = (\Psi_1, \Psi_2)$  
defines a bivariate random closed set by the supports
\[
{\rm{supp}}(\Psi_i) =
 \bigcap_{n=0}^\infty {\rm{cl}}( \{ x_j \in \Q^d: \Psi_i( B\left(
  x_j, {1}/{n}\right) ) > 0 \} )
\]
where $B\left(x_j, {1}/{n}\right)$ is the closed ball around $x_j$ 
with radius $1/n$ and ${\rm{cl}}(B)$ is the topological closure of 
the Borel set $B$. In other words, if $x \in {\rm{supp}}(\Psi_i)$,
then every ball that contains $x$ has strictly positive $\Psi_i$-mass.
By \cite[Prop.~8.16]{Molc05}, the supports are well-defined random 
closed sets whose joint distribution is uniquely determined by that 
of the random measures. 

Indeed, Ayala {\em et al.\/} \cite{Ayal91} proved the following result.

\begin{thm}
\label{t:ayala}
Let $X = (X_1, \dots, X_n)$ be a multivariate random closed set. Then the 
distribution of $X$ is recoverable from $\Psi = (\ell(X_1 \cap \cdot),
\cdots, \ell(X_n \cap \cdot) )$ if and only if $X$ is distributed as the 
(random) support of $\Psi$.
\end{thm}

From now on, assume that $X$ is stationary. Then the {\em hitting intensity\/}
\cite{StoyOhse82} is defined as
\[
T_{12}(t) = \EE \left[ \frac{1}{\ell(B)} 
  \int_B  1\{ X_2 \cap B(x,t) \neq \emptyset \} \, d\Psi_1(x) \right] 
\]
where $B$ is any bounded Borel set of positive volume $\ell(B)$ and
$B(x,t)$ is the closed ball centred at $x\in\R^d$ with radius $t\geq 0$.
The definition does not depend on the choice of $B$.
The hitting intensity is similar in spirit to another classic statistic,
the {\em empty space function\/} \cite{Math75} defined by
\[
F_2(t) = \PP( X_2 \cap B(x,t) \neq \emptyset ).
\]
 The related {\em cross spherical contact 
distribution\/} can be defined as
\[
H_{12}(t) = \PP( X_2 \cap B(x,t) \neq \emptyset | x \in X_1 )
\]
in analogy to the classical univariate definition \cite{Chiu13}.
Again, the definitions do not depend on the choice of $x\in\R^d$ due to
the assumed stationarity.

In order to relate $T_{12}$ and $F_2$ to our $J_{12}$ statistic, we need the 
concept of `scaling'. Let $s>0$ be a scalar. Then the scaling of $X$ by
$s$ results in $sX = (sX_1, s X_2)$ where $sX_i = \{ sx: x \in X_i \}$.

\begin{thm}
\label{t:rcs}
Let $X = (X_1, X_2)$ be a stationary bivariate random closed set 
with strictly positive volume fractions $p_1(0,i) = \PP( 0 \in X_i)$, 
$i=1,2$. Then the associated random coverage measure $\Psi$ is 
coverage-reweighted moment stationary and the following hold.
\begin{enumerate}
\item The cross statistics are 
\begin{eqnarray*}
K_{12}(t) & =  & 
\frac{ \EE \left( \ell( X_2 \cap B(0,t) ) | 0 \in X_1 \right ) }{p_1(0,2) }; \\
J_{12}(t) & = &
\frac{ 
\EE \left(
 1\{ 0 \in X_1 \} 
\exp\left[ -  \ell( X_2 \cap B(0,t)) / p_1(0,2) \right]  \right)
}{
p_1(0,1) \, \EE \left(
   \exp\left[ -  \ell(X_2 \cap B(0,t)) / p_1(0,2)\right ] \right) 
}.
\end{eqnarray*}
\item Use a subscript $sX$ to denote that the statistic is evaluated
for the scaled random closed set $sX$ and let $u_t^0$ be as in 
Theorem~\ref{t:j12}. Then
\[
 \lim_{s\to\infty} L^{(0,1)} (s^d u_{t}^0 ) =
\frac{1 - T_{12}(t)}{p_1(0,1)}
\]
and, for $t>0$,
\[
\lim_{s\to\infty} J_{12; sX}(st) = \frac{
\PP( X_2 \cap B(0,t) = \emptyset | 0 \in X_1 )
}{
\PP( X_2 \cap B(0,t) = \emptyset )
} =
\EE \left(
\frac{ 1\{ 0 \in X_1 \} }{ p_1(0,1) } | X_2 \cap B(0,t) = \emptyset \right)
\]
whenever $\PP(X_2 \cap B(0,t) = \emptyset) \neq 0$.
\end{enumerate}
\end{thm}

In words, the scaling limit of the cross $J$-function compares the 
empty space function to the cross spherical contact distribution.

\

\begin{proof}
First note that
\[
\mu^{(k)}( (B_1 \times \{ i_1 \}) \times \cdots \times (B_k \times \{ i_k \}) )
 = 
\EE\left( 
 \ell(X_{i_1} \cap B_1) \times \cdots \times \ell(X_{i_k} \cap B_k) 
\right), 
\]
which, by \cite[(4.14)]{Molc05} is equal to
\[
   \int_{B_1} \cdots \int_{B_k} \PP( x_1 \in X_{i_1}; \dots; x_k \in X_{i_k} ) \,
dx_1 \cdots dx_k.
\]
Here, $k\in\N$ and $B_1, \dots, B_k$ are Borel subsets of $\R^d$.
Hence, $\Psi$ admits moment measures of all orders and
the probabilities
\(
\PP( x_1 \in X_{i_1}; \dots; x_k \in X_{i_k} ) = 
p_k((x_1, i_1), \dots, (x_k, i_k))
\)
define the coverage functions. By assumption $p_1$ is bounded away
from zero, so the stationarity of $X$ implies that $\Psi$ is 
coverage-reweighted moment stationary.

Since by \cite[p.~288]{Chiu13}, the Palm distribution amounts to conditioning
on having a point of the required component at the origin, the expression
for the cross $K$-function follows from Lemma~\ref{l:Kpalm}. 

To see the effect of scaling on $J_{12}$, observe that since 
\[
\PP( x_1 \in sX_{i_1}; \cdots x_k \in sX_{i_k} ) = 
\PP( x_1 /s \in X;_{i_1} \cdots x_k/s \in X_{i_k} ),
\]
the $k$-point coverage probabilities of $sX$ are related to
those of $X$ by $p_{k; sX}( (x_1, i_1), \dots, (x_k, i_k) ) $ 
$ = p_X( (x_1/s, i_1), \dots, (x_k/s, i_k) )$.  Similarly, 
$\xi_{k; sX}((x_1, i_1), \dots, (x_k, i_k) ) = $
$\xi_{k;X}( (x_1/s, i_1), \dots, $ 
$ (x_k/s, i_k))$
and consequently $J^{(k)}_{12;sX}(t)$ $ = s^{dk} J^{(k)}_{12;X}(t/s)$. 
Also scaling the balls $B(0,t)$ by $s$ to fix the coverage fraction,
one obtains
\(
J^{(k)}_{12; sX}(st) = s^{dk} J^{(k)}_{12;X}(t).
\)
The numerator in the expression of $J_{12}$ in terms of Laplace 
functionals (cf.\ Theorem ~\ref{t:j12}) after such scaling reads as 
follows. Define
\[
u_{st; sX} =  \frac{ 1\{ (x,i) \in B(0, st) \times \{ 2 \} \} 
}{ p_{1; sX}(x, i) } = 
\frac{ 1\{ (x/s ,i) \in B(0, t) \times \{ 2 \} \} 
}{
p_{1; X}( x / s, i )}.
\]
Then 
\[
L_{sX}^{(0,1)}(u_{st;sX})  =  \EE \left( \exp\left[
  - \int_{B(0,st)} \frac{1\{ x \in sX_2\} }{ p_{1;sX}( x,2)} \, dx \right] |
0 \in sX \right) =  L_X^{(0,1)} (s^d u_{t; X} ).
\]
For $t>0$, as $s\to \infty$
\[
 L_X^{(0,1)} (s^d u_{t;X} ) \to
 \PP( X_2 \cap B(0,t) = \emptyset | 0 \in X_1 ) 
\]
by the monotone convergence theorem. 

Turning to $T_{12}(t)$, note that
\[
\EE\left[ 
\frac{1}{\ell(B)} \int_B 1\{ X_2 \cap B(x,t) \neq \emptyset; x \in X_1 \} \, dx
\right] = 
\frac{1}{\ell(B)} \int_B \PP( X_2 \cap B(x,t) \neq \emptyset; x \in X_1 ) \,dx
\]
by Robbins' theorem. Since the volume fractions are strictly positive, we 
may condition on having a point at any $x\in\R^d$, so that
\[
\PP( X_2 \cap B(x,t) \neq \emptyset; x \in X_1 )  =
\PP( X_2 \cap B(0,t) \neq \emptyset | 0 \in X_1 ) \PP( 0 \in X_1)
\]
upon using the stationarity of $X$. We conclude that 
\(
 L_X^{(0,1)} (s^d u_{t;X} ) \to ( 1 - T_{12}(t) ) / p_1(0,1)
\)
as claimed.

Finally, consider the effect of scaling on the denominator in 
(\ref{e:j12}). Now,
\[
L_{sX}(u_{st;sX})  =  
 \EE\left[ \exp( - \ell( sX_2 \cap B(0, st ) ) / p_1( 0, 2 )  
 \right] 
 =  
 L_X (s^d u_{t;X} ).
\]
For $t>0$, 
\[
\lim_{s\to\infty} L_X(s^d u_{t;X} ) = \PP( X_2 \cap B(0,t) = \emptyset )
\]
by the monotone convergence theorem. Combining numerator and
denominator, the theorem is proved.
\end{proof}

The case $t=0$ is special. Indeed, both the spherical contact distribution
and empty space function may have a `nugget' at the origin. In contrast,
$J_{12}(0) \equiv 1$.

Before specialising to germ-grain models, let us make a few remarks. 
First, note that the moment measures of $\Psi$ have a nice interpretation. 
Indeed, by Fubini's theorem, the $k$-point coverage function coincides 
with the $k$-point coverage probabilities of the underlying random closed
set. Moreover, since  $\mu^{(k)}((B \times \{ 1,\dots, n \})^k) \leq 
( n\ell(B) )^k$, the Zessin condition holds, cf.\ Theorem~\ref{t:zessin}.

Secondly, if $X_1$ and $X_2$ are independent, $J_{12}(t) \equiv 1$.
More generally, if $\ell( X_2 \cap B(0,t) )$ and $1\{ 0 \in X_1 \}$ are 
negatively quadrant dependent, $J_{12}(t) \geq 1$. If the two random 
variables are positively quadrant dependent, then $J_{12}(t) \leq 1$. 
A similar interpretation holds for the cross $K$-function: 
 if $\ell( X_2 \cap B(0,t) )$ and $1\{ 0 \in X_1 \}$ are 
negatively correlated, $K_{12}(t) \leq \kappa_d t^d$;
if the two random variables are positively correlated, then $K_{12}(t) 
\geq \kappa_d t^d$. 

\paragraph{Germ-grain models}

Let $N = (N_1, N_2)$ be a stationary bivariate point process. 
Placing closed balls of radius $r>0$ around each of the points 
defines a bivariate random closed set 
\[
(X_1, X_2) = ( U_r(N_1), U_r(N_2)),
\]
where, for every locally finite configuration $\phi\subseteq \R^d$
\[
U_r(\phi) = \bigcup_{x\in \phi} B(x, r).
\]

\begin{thm}
Let $N = (N_1, N_2)$ be a stationary bivariate point process and 
$X$ the associated germ grain model for balls of radius $r>0$. 
Write, for $x\in\R^d$, $t_1, t_2 \in \R^+$,
\[
F_N(t_1, t_2 ; x) = \PP ( d(0, X_1) \leq t_1; d(x, X_2) \leq t_2 )
\]
for the joint empty space function of $N$ at lag $x$ and let
$F_{N_i}$ be the marginal empty space function of $N_i$, $i=1,2$. 
If $F_{N_i}(r) > 0$ for $i=1, 2$, the random coverage measure $\Psi$ 
of $X$ is coverage-reweighted moment stationary with
\[
K_{12}(t)  =   \frac{1}{F_{N_1}(r) \, F_{N_2}(r) } \,
\int_{B(0,t)} F_N( r, r; x) \, dx 
\]
and, for $t>0$,
\[
\lim_{s\to\infty} J_{12; sX}(st)  = 
\frac{ F_{N_1}(r) - F_N(r, r+t; 0)}{ F_{N_1}(r) ( 1 - F_{N_2}(r+t) ) }
\]
whenever $F_{N_1}(r) > 0$ and $F_{N_2}(r+t) < 1$.
\end{thm}

Hence, the cross statistics of the germ-grain model can be
expressed entirely in terms of the joint empty space function
of the germ processes; the radius of the grains translates itself
in a shift.

\

\begin{proof}
Since the coverage probabilities 
\[
p_1(0,i) = \PP( 0 \in X_i ) = \PP( d(0, N_i) \leq r ) = F_{N_i}(r)
\]
are strictly positive by assumption, Theorem~\ref{t:rcs} implies
that $\Psi$ is coverage-reweighted moment stationary. 
By stationarity,
\[
K_{12}(t) = \frac{1}{F_{N_1}(r) \, F_{N_2}(r) }
 \int_{B(0,t)} \PP( 0\in X_1; x\in X_2)  \, dx.
\]
The observation that
\[
\PP( 0\in X_1; x\in X_2) = \PP( d(0, N_1) \leq r; d(x, N_2) \leq r )
= F_N(r, r; x)
\]
which implies the claimed expression for the cross $K$-statistic.
Furthermore,
\[
 \PP( {X_2} \cap B(0,t) \neq \emptyset ) = 
\PP( d(0, N_2) \leq r + t ) = F_{N_2}(r+t)
\]
and
\begin{eqnarray*}
\frac{1-T_{12}(t)}{p_1(0,1)} & =  &
\PP( X_2 \cap B(0,t) = \emptyset | 0 \in X_1 ) 
=  \frac{ \PP(N_1 \cap B(0,r) \neq \emptyset; N_2 \cap B(0,r+t) = \emptyset)
}{
\PP( N_1 \cap B(0,r) \neq \emptyset )} \\
& = & \frac{F_{N_1}(r) - F_N(r, r+t; 0)}{F_{N_1}(r) }
\end{eqnarray*}
can be expressed in terms of the joint empty space function of
$(N_1, N_2)$. The claim for the scaling limit of $J_{12}$ follows
from Theorem~\ref{t:rcs}.
\end{proof}

For the special case $t=0$, note that although $J_{12}(0) = 1$, in the limit
\(
F_{N_1}(r) - F_N(r, r; 0) 
\)
is not necessarily equal to 
\(
 F_{N_1}(r) - F_{N_1}(r) F_{N_2}(r)
\)
unless $N_1$ and $N_2$ are independent.

The stationarity assumption seems required. Consider for example
a Boolean model \cite{Molc97} obtained as the union set $X$ of 
closed balls of radius $r>0$ centred at the points of a Poisson 
process with intensity function $\lambda(\cdot)$. For this model, 
first and second order $k$-point coverage functions are given by
\begin{eqnarray*}
p_1(x) & = & 1 - \exp\left[ - \int \lambda(z) \, 
 1\{ z \in B(x,r) \} \, dz \right]; \\
p_2(x, y) & = & p_1(x) + p_1(y) - 1 + \exp\left[
 - \int \lambda(z) \, 1\{ z \in B(x,r) \cup B(y,r) \} \, dz 
\right].
\end{eqnarray*}
Hence
\(
 \xi_2( x, y) 
\)
is not necessarily invariant under translations
contrary to the claim in \cite{Gall14}. 


Even in the stationary case, that is, for constant $\lambda(\cdot)$,
the Laplace transform 
\(
L( u_t^0 ) = \EE \exp[ - \ell( X \cap B(0,t) ) / p_1(0) ]
\)
is intractable, being the partition function of an area-interaction 
process with interaction parameter $\log\gamma =  1 / p_1(0)$ and range 
$r$ observed in the ball $B(0,t)$  \cite{LiesBadd96}.

\subsection{Random field models}

Inhomogeneity may be introduced into the coverage measure associated
to a random closed set by means of a random weight function.
Let $X = (X_1, X_2)$ be a bivariate random closed set and 
$\Gamma = (\Gamma_1, \Gamma_2)$ a bivariate random field
taking almost surely non-negative values. Suppose that $X$ and $\Gamma$
are independent and set
\(
\Psi = (\Psi_1, \Psi_2) 
\)
where 
\begin{equation}
\label{e:Ballani}
\Psi_i(B) = \int_{B}  \Gamma_i(x) 1\{ x\in X_i \} \, dx.
\end{equation}
The univariate case was dubbed a random field model by Ballani
{\em et al.\/} \cite{Ball12} for which, under the assumption that both $X$ and
$\Gamma$ are  stationary, \cite{Koub16} employed the $R_{12}$-statistic
for testing purposes.

\begin{thm}
Let (\ref{e:Ballani}) be a bivariate random field model and suppose that
$\Gamma$ admits a continuous version and that its associated random measure
is coverage-reweighted moment stationary. Furthermore, assume that
$X$ is stationary and has strictly positive volume fractions.
Then the random field model is coverage-reweighted moment stationary 
and, writing $c_{12}^X$ respectively $c_{12}^\Gamma$ for the
coverage-reweighted cross covariance functions of $X$ and $\Gamma$, 
the following hold:
\begin{eqnarray*}
K_{12}(t) & = & \int_{B(0,t)} ( c_{12}^X(0,x) + 1 ) \, 
( c_{12}^\Gamma(0,x) + 1 ) \, dx; \\
J_{12}(t) & = & \frac{
\EE \left( \Gamma_1(0) 
\exp\left[
- \frac{1}{\PP( 0 \in X_2 )} \int_{B(0,t)\cap X_2}
  \frac{\Gamma_2(x)}{\EE\Gamma_2(x)} \, dx \right]
| 0 \in X_1 \right)
}{
\EE \Gamma_1(0) \, \EE\exp\left[
- \frac{1}{\PP( 0 \in X_2 )} \int_{B(0,t)\cap X_2}
  \frac{\Gamma_2(x)}{\EE\Gamma_2(x)} \, dx 
\right]
}.
\end{eqnarray*}
\label{t:random-field}
\end{thm}

\begin{proof}
First, with $p_k^X$ for the $k$-point coverage probabilities of $X$,
\[
\EE\left[ 
  \Psi_1(B_1) \cdots \Psi_1(B_k) \Psi_2(B_{k+1}) \cdots \Psi_2(B_{k+l})
\right] =
\]
\[
\EE\left[ \int_{B_1} \cdots \int_{B_k} \int_{B_{k+1}} \cdots \int_{B_{k+l}} 
\left( \prod_{i=1}^k  1\{ x_i \in X_1 \} \Gamma_1(x_i) \, dx_i \right)
\left( \prod_{i=1}^l 1\{ y_i \in X_2 \} \Gamma_2( y_i) \, dy_i \right)
 \right] = 
\]
\[
 \int_{B_1} \cdots \int_{B_k} \int_{B_{k+1}} \cdots \int_{B_{k+l}}  
p^X_{k+l}( (x_1, 1), \dots, (x_k,1), (x_{k+1}, 2), \dots, (x_l, 2) ) \times
\]
\[
\times
 \EE\left[ \prod_{i=1}^k \Gamma_1(x_i) \prod_{i=1}^l \Gamma_2(y_i) \right]
   dx_1 \cdots dx_k \, dy_1 \cdots dy_l
\]
by the monotone convergence theorem and the independence of $X$ and 
$\Gamma$ (recalling the moment measures are locally finite). 
Hence, $\mu^{(k+l)}$ is absolutely continuous 
and its Radon--Nikodym derivative $p_{k+l}$ satisfies
\[
\frac{p_{k+l}( (x_1, 1), \dots, (x_k,1), (x_{k+1}, 2), \dots, (x_{k+l}, 2) )}{
p_1(x_1, 1) \cdots p_1(x_k,1)\, p_1( x_{k+1},2) \cdots p_1(x_{k+l}, 2) } = 
\]
\[
= \frac{p^X_{k+l}( (x_1, 1), \dots, (x_k,1), (x_{k+1}, 2), \dots, (x_{k+l}, 2) )}{
p^X_1(x_1, 1) \cdots p^X_1(x_k,1) \, p^X_1( x_{k+1},2) \cdots p^X_1(x_{k+l}, 2) } 
\, \frac{
 \EE\left[ \prod_{i=1}^k \Gamma_1(x_i) \prod_{i=1}^l \Gamma_2(y_i) \right] }{
 \prod_{i=1}^k \EE \Gamma_1(x_i) \prod_{i=1}^l \EE \Gamma_2(y_i) }.
\]
Here $p^X_{k+l}$ denotes the $k+l$-point coverage probability of $X$.
Since $X$ is stationary and $\Gamma$ coverage-reweighted moment stationary, 
translation invariance follows. Moreover, the function
\[
p_1(x, i) = p_1^X(x,i) \, \EE \Gamma_i(x) = p_1^X(0,i) \, \EE \Gamma_i(x)
\]
is bounded away from zero because $X$ has strictly positive volume fractions
and $\Gamma$ is coverage-reweighted moment stationary by assumption.

For $k=2$ we have
\[
\xi_2( (x,1), (y,2) ) = \frac{ p_2^X( (x,1), (y,2) ) }{p_1^X( x,1) \, p_1^X(y,2)}
\frac{ \EE\left[ \Gamma_1(x) \Gamma_2(y) \right]}{\EE \Gamma_1(x) \,
\EE \Gamma_2(y) }  - 1
\]
from which the claimed form of the cross $K$-statistic follows.
For the cross $J$-statistic, one needs the 
Palm distribution. By the Campbell--Mecke formula, for any 
Borel set $A\subseteq \R^d$,  $i=1,2$, and any measurable $F$,
\[
\int_A \PP^{(x,i)}(F)\,  p_1( x,i) \, dx =
\EE\left[
\int_{A\cap X_i} 1_F(\Psi) \Gamma_i(x)  \, dx 
\right] =
 \int_A \frac{\EE\left[ 1_F(\Psi) \Gamma_i(x) |  x \in X_i \right] }{
  \EE \Gamma_i(x) } \, p_1(x,i) \, dx
\]
by Fubini's theorem. Therefore, for $p_1$-almost all $x$ and  $i = 1,2$
\[
\PP^{(x,i)}(F) = \frac{ \EE\left[ \Gamma_i(x) 1_F(\Psi) | x\in X_i \right] }{
\EE \Gamma_i(x) }
\]
and the proof is complete.
\end{proof}

Note that if the covariance functions of both the random closed set
$X$ and the random field $\Gamma$ are non-negative, $K_{12}(t) \geq
\kappa_d t^d$; if there is non-positive correlation, $K_{12}(t) \leq 
\kappa_d t^d$. Similarly, if the random variables
$\Gamma_1(0) 1\{ 0 \in X_1 \}$ and 
\[
 \int_{B(0,t)\cap X_2} \frac{\Gamma_2(x)}{\EE\Gamma_2(x)} \, dx 
\]
are positively quadrant dependent, $J_{12}(t) \leq 1$ and, reversely,
$J_{12}(t) \geq 1$ when they are negatively quadrant dependent.

\paragraph{Log-Gaussian random field model}

A flexible choice is to take $\Gamma_i = e^{Z_i}$ for some bivariate 
Gaussian random field $Z = (Z_1, Z_2)$ with mean functions $m_i$, $i=1,2$
and (valid) covariance function matrix $(c_{ij})_{i,j\in \{1, 2\}}$. 
Since $\Psi$ involves integrals 
over $\Gamma$,  conditions on $m_i$ and $c_{ii}$ are needed. Therefore, 
we shall assume that $m_1$ and $m_2$ are continuous, bounded functions,
for example taking into account covariates.
For the covariance function, sufficient conditions
are given in \cite[Theorem~3.4.1]{Adle81}. Further details and
examples can be found in \cite{Moll98} or in \cite[Section~5.8]{MollWaag04}. 

\begin{thm}
Consider a bivariate random field model for which 
$\Gamma$ is log-Gaussian with bounded continuous mean functions and 
translation invariant covariance functions $\sigma^2_{ij} r_{ij}( \cdot )$ 
such that $\Gamma$ admits a continuous version.  Furthermore, assume that 
$X$ is stationary and has strictly positive volume fractions.
Then the random field model is coverage-reweighted moment stationary 
and the following hold. The cross $K$-function is equal to
\[
K_{12}(t)  =   \int_{B(0,t)} ( 1 + c_{12}^X(0,x) ) \,
\exp\left[ \sigma_{12}^2 \, r( x ) \right] dx
\]
where $c_{12}^X$ is the coverage-reweighted cross covariance function
of $X$;  the cross $J$-statistic reads
\begin{eqnarray*}
J_{12}(t) & = & \frac{ \EE \left( \exp\left[
Y_1(0) - \frac{1}{\PP(0\in X_2)}
\int_{B(0,t) \cap X_2} e^{Y_2(x)} dx 
\right] | 0 \in X_1 \right) }{
\EE \exp\left[
- \frac{1}{\PP(0\in X_2)}
\int_{B(0,t) \cap X_2} e^{Y_2(x)} dx 
\right]  }
\\
& = & \frac{ \EE \left[ \exp\left( 
- \frac{1}{\PP(0\in X_2)}
\int_{B(0,t) \cap X_2} e^{Y_2(x) + \sigma^2_{12} r_{12}(x) } dx 
\right] | 0 \in X_1 \right) }{
\EE \exp\left[
- \frac{1}{\PP(0\in X_2)}
\int_{B(0,t) \cap X_2} e^{Y_2(x)} dx 
\right]  }
\end{eqnarray*}
where $Y_i(x) = Z_i(x) - m_i(x) - 0.5 \sigma^2_{ii}$.
\end{thm}

\begin{proof}
For a log-Gaussian random field model,
\[
 \EE \exp\left[ \sum_{i=1}^k Z_1(x_i) + \sum_{i=1}^l Z_2( y_i) \right]
= \exp\left[ \sum_{i=1}^k m_1(x_i) + \sum_{i=1}^l m_2(y_i) +
  \frac{k}{2}  \sigma^2_{11} +   \frac{l}{2}  \sigma^2_{22}
\right] \times
\]
\[
\exp\left[
 \sigma^2_{11}  \sum_{1\leq i < j \leq k} r_{11}(x_j-x_i) )
+ \sigma^2_{22}  \sum_{1\leq i < j \leq l} r_{22}(y_j-y_i) )
+ \sigma^2_{12} \sum_{1\leq i \leq k} \sum_{1\leq j \leq l} r_{12}(y_j-x_i) 
\right]
\]
so that, with notation  as in the proof of Theorem~\ref{t:random-field},
$\mu^{(k+l)}$ is absolutely continuous and its Radon--Nikodym
derivative $p_{k+l}$ satisfies
\[
\frac{p_{k+l}( (x_1, 1), \dots, (x_k,1), (x_{k+1}, 2), \dots, (x_{k+l}, 2) )}{
p_1(x_1, 1) \cdots p_1(x_k,1) \, p_1( x_{k+1},2) \cdots p_1(x_{k+l}, 2) } = 
\]
\[
= \frac{p^X_{k+l}( (x_1, 1), \dots, (x_k,1), (x_{k+1}, 2), \dots, (x_{k+l}, 2) )}{
p^X_1(x_1, 1) \cdots p^X_1(x_k,1) \, p^X_1( x_{k+1},2) \cdots p^X_1(x_{k+l}, 2) } 
\times
\]
\[
\times
\exp\left[
 \sigma^2_{11} \sum_{1\leq i < j \leq k} r_{11}(x_j-x_i) 
+ \sigma^2_{22} \sum_{1\leq i < j \leq l} r_{22}(y_j-y_i)
+ \sigma^2_{12} \sum_{1\leq i \leq k} \sum_{1\leq j \leq l} r_{12}(y_j-x_i) 
\right].
\]
Since $X$ is stationary, translation invariance follows. 

For $k=1$ and $k=2$ we have
\[
p_1(x, i) = p_1^X(0,i) \exp\left[ m_i(x) + \sigma^2_{ii} / 2 \right]
\]
and
\[
\xi_2( (x,1), (y,2) ) = \frac{ p_2^X( (x,1), (y,2) ) }{p_1^X( x,1)\, p_1^X(y,2)}
\exp\left[ \sigma_{12}^2 r_{12}( y-x ) \right] - 1.
\]
The function $p_1(x,i)$ is bounded away from zero since $X$ has 
strictly positive volume fractions and the $m_i$ are bounded.
The form of the cross $K$-statistic follows from that of $\xi_2$ and
the first  expression for $J_{12}(t)$ is an immediate consequence of 
Theorem~\ref{t:random-field}.

Finally, consider the ratio of $p_{1+k+l}( (a,1), (x_1, 1), \dots, (x_k, 1),
(x_{k+1},2), \dots, (x_{k+l}, 2) ) $ and $p_1(a,1) \prod_{i=1}^{k+l} p_1(x_i, i_i)$,
which can be written as
\[
\frac{ \PP( x_i \in X_1, i=1, \dots, k; x_{k+i} \in X_2, i= 1, \dots, l  | 
a \in X_1 )} {
\prod_{i=1}^k \PP( x_i \in X_1) \prod_{i=1}^l \PP(x_{k+i} \in X_2 ) } \times
\]
\[
\frac{
p^\Gamma_{k+l}( (x_1, 1), \dots, (x_k, 1), (x_{k+1}, 2), \dots, (x_{k+l}, 2) )
}{
\prod_{i=1}^k p^\Gamma_{1} ( x_i, 1 ) \prod_{i=1}^l p^\Gamma_{1} ( x_{k+i}, 2) }
\times 
\prod_{i=1}^k e^{ \sigma^2_{11} r_{11}( x_i - a) }
\prod_{i=1}^l e^{ \sigma^2_{12} r_{12}( x_{k+i} - a) }.
\]
Hence $L^{(a,1)}(u_t^a)$ (cf.\ Theorem~\ref{t:j12}) 
becomes the Laplace functional $L$ evaluated for the function 
\[
\tilde u_t^a( x, i) =  1\{ ( x,i) \in B(a,t) \times \{ 2 \} \} \,
\exp\left[ \sigma^2_{12} r_{12}(x-a) \right] / p_1(x, 2)
\]
after conditioning on $a\in X_1$, an observation which completes the proof.
\end{proof}

In the context of a point process, \cite{Coeu15} prove the 
stronger result that the Palm distribution of a log-Gaussian Cox
process is another log-Gaussian Cox process.

\paragraph{Random thinning field model}

Consider the following random field model \cite{Digg83}
with inter-component dependence modelled by means of
a (deterministic) non-negative function $r_i(x)$, $i=1, 2$,
on $\R^d$ such that $r_1 + r_2 \equiv 1$. 
Let $\Gamma_0$ be a non-negative random field and assume that the 
components $\Gamma_i(x) = r_i(x) \Gamma_0(x)$ 
are integrable on bounded Borel sets. As before, $X$ is a stationary 
bivariate random closed set and a random measure is defined through 
(\ref{e:Ballani}). Heuristically speaking, the $r_i(x)$ can be thought 
of as location dependent retention probabilities for $X_i$. 

For the model just described, 
\[
1 + c_{12}^\Gamma(0,x) = \frac{\EE\left[ \Gamma_0(0) \Gamma_0(x) \right]}{
\EE \Gamma_0(0) \, \EE \Gamma_0(x) } = 1 + c^{\Gamma_0}(0, x)
\]
and similarly for higher orders so that $\Gamma$ is coverage-reweighted 
moment stationary precisely when $\Gamma_0$ is.  Hence 
Theorem~\ref{t:random-field} holds with the $\Gamma_i$ 
replaced by $\Gamma_0$. 

\section{Estimation}
\label{S:estimation}

For notational convenience, introduce the random measure 
$\Phi = ( \Phi_1, \Phi_2)$ defined by 
\[
\Phi_i(A) = \int_A  \frac{1}{p(x,i)} \, d\Psi_i(x)
\]
for Borel sets $A\subseteq \R^d$.

\begin{thm}
Let $\Psi = (\Psi_1, \Psi_2)$ be a coverage-reweighted moment stationary
bivariate random measure that 
is observed in a compact set $W\subseteq \R^d$ whose erosion $W_{\ominus t}
= \{ w\in W: B(w, t) \subseteq W \}$ has positive volume 
$\ell(W_{\ominus t}) > 0$. Then, under the assumptions of 
Theorem~\ref{t:j12}, 
\begin{equation}
\label{e:fhat}
\widehat{ L_2(t)} = 
  \frac{1}{\ell( W_{\ominus t})} 
\int_{W_{\ominus t}} e^{-\Phi_2( B(x,t) )} \, dx
\end{equation}
is an unbiased estimator for $L(u_t^0)$,
\begin{equation}
\label{e:khat}
\widehat{K_{12}(t)} = 
  \frac{1}{\ell( W_{\ominus t})} 
\int_{W_{\ominus t}} \Phi_2( B(x,t) ) \, d\Phi_1(x)
\end{equation}
is an unbiased estimator for $K_{12}(t)$ and
\begin{equation}
\label{e:dhat}
\widehat{L_{12}(t)} = 
  \frac{1}{\ell( W_{\ominus t})} 
\int_{W_{\ominus t}} e^{-\Phi_2( B(x,t) )} \, d\Phi_1(x)
\end{equation}
is unbiased for $L^{(0,1)}(u_t^0)$.
\end{thm}

\begin{proof}
First, note that for all $x\in W_{\ominus t}$ the mass
\(
\Phi_2(B(x,t)) 
\)
can be computed from the observation since $B(x,t) \subseteq W$. 
Moreover, 
\[
\EE \left[ e^{- \Phi_2( B(x,t) ) } \right] = 
L( 1_{B(x,t) \times \{ 2 \} }(\cdot) / p_1( \cdot) )
\]
regardless of $x$ by an appeal to Theorem~\ref{t:j12}.
Consequently, (\ref{e:fhat}) is unbiased. 

Turning to (\ref{e:dhat}), by (\ref{e:CM}) with 
\[
g( (x,i), \Psi) = \frac{1_{ W_{\ominus t} \times \{1\} }(x,i)}{p_1(x,i) }
\exp[ -\Phi_2(B(x,t) ) ] 
\]
we have
\[
\ell(W_{\ominus t}) \EE\widehat{L_{12}(t)} = 
 \int_{W_{\ominus t}} 
 \frac{ L^{(x,1)}( 1_{B(x,t) \times \{ 2 \} }(\cdot) / p_1( \cdot) )}{ p_1(x,1) } 
    p_1(x,1) \, dx.
\]
Since $L^{(x,1)}( 1_{B(x,t) \times \{ 2 \} }(\cdot) / p_1( \cdot ))$ does
not depend on $x$ by Theorem~\ref{t:j12}, the estimator is unbiased. 
The same argument for 
\[
\tilde g( (x,i), \Psi) = \frac{1_{ W_{\ominus t} \times \{1\} }(x,i)}{p_1(x,i) }
\Phi_2(B(x,t) ) 
\]
proves the unbiasedness of $\widehat{K_{12}(t)}$.
\end{proof}

A few remarks are in order. 
In practice, the integrals will be approximated by Riemann sums.
Moreover, in accordance with the Hamilton principle \cite{Stoy2},
the denominator $\ell(W_{\ominus t})$ in $\widehat{K_{12}(t)}$ and
$\widehat{L_{12}(t)}$ can be replaced by $\Phi_1(W_{\ominus t})$.
Finally, we assumed that the coverage function is known. If 
this is not the case, a plug-in estimator may be used.

\section{Illustrations}
\label{S:examples}

In this section, we illustrate the use of our statistics on simulated
realisations of some of the models discussed in Section~\ref{S:theo-ex}.

\paragraph{Widom--Rowlinson mixture model}

First, consider the Widom--Rowlinson mixture model \cite{WidoRowl70}
defined as follows. Let $(N_1, N_2)$ be a bivariate point process whose 
joint density with respect to the product measure of two independent 
unit rate Poisson processes is 
\[
f_{\rm mix}(\phi_1, \phi_2 ) \propto
 \beta_1^{|\phi_1|} \beta_2^{|\phi_2|} 1\{ d(\phi_1, \phi_2) > r\},
\]
writing $|\cdot|$ for the cardinality and $d(\phi_1, \phi_2)$ for 
the smallest distance between a point of $\phi_1$ and one of $\phi_2$. 
In other words, points of different components are not allowed to be within
distance $r$ of one another.

A sample from this model can be obtained by coupling from the past
\cite{Hagg99,KendMoll00,LiesStoi06}. For the picture displayed in 
Figure~\ref{F:WR}, we used the {\tt mpplib} library \cite{mpplib} to 
generate a realisation with $\beta_1 = \beta_2 = 1$ and $r=1$ on 
$W = [0,10]\times[0,20]$. To avoid edge effects, we sampled on 
$[-1,11] \times[-1,21]$ and clipped the result to $W$. Placing 
closed balls of radius $r/2$ around each of the points yields a
bivariate random closed set, the Widom--Rowlinson germ grain model.
Note that 
\[
U_{r/2}(\phi_1) \cap U_{r/2}(\phi_2) = \emptyset,
\]
so that there is negative association between the two components.

The estimated cross statistics are shown in Figure~\ref{F:WRstat}.
The graph of $\widehat{L_{ij}(t)}$ lies above that of $\widehat{L_{j}(t)}$ 
reflecting the inhibition between the components. The graph of
 $\widehat{K_{ij}(t)}$ lies below that of the function $t\to \pi t^2$,
which confirms the negative correlation between the components.

\paragraph{Dual Widom--Rowlinson mixture model}

The dual Widom--Rowlinson mixture model is based on
\[
f_{\rm mix}(\phi_1, \phi_2 ) \propto
\beta_1^{n(\phi_1)} \beta_2^{n(\phi_2) } \,
1\{ \phi_2 \subseteq U_r(\phi_1) \}.
\]
Since the points of the second component lie in $U_r(\phi_1)$,
that is, within distance $r$ of a point from the first component,
the model exhibits positive association. Placing balls of radius
$r/2$ around the components yields a germ-grain model. 

Exact samples from this model can be obtained in two steps. First, 
generate an area-interaction point process with parameter $\beta_1$ and 
$\gamma = e^{-\beta_2}$ using coupling from the past \cite{KendMoll00}
by the {\tt mpplib} library \cite{mpplib}. Then, conditionally on 
the first component being $\phi_1$, generate a Poisson process of 
intensity $\beta_2$ and accept only those points that fall in 
$U_r(\phi_1)$.  Figure~\ref{F:dualWR} shows a realisation with 
$\beta_1 = \beta_2 = 1/4$ and $r=1$ on $W = [0,10]\times[0,20]$. 
To avoid edge effects, we sampled on $[-1,11] \times[-1,21]$ and 
clipped the result to $W$. 

The estimated cross statistics are shown in Figure~\ref{F:dualstat}.
The graph of $\widehat{L_{ij}(t)}$ lies below that of $\widehat{L_{j}(t)}$ 
reflecting the attraction between the components. The graph of
 $\widehat{K_{ij}(t)}$ lies above that of the function $t\to \pi t^2$,
which confirms the positive correlation between the components.

\paragraph{Boolean model marked by linked log-Gaussian field}

Our last illustrations concern random field models based on 
Gaussian random fields. Thus, let $\Gamma_0$ be a Gaussian random 
field with mean function $m(\cdot)$ and exponential covariance function 
\begin{equation}
\sigma^2 \exp[-\beta || x - y || ].
\label{e:cexp}
\end{equation}
The package {\tt fields} \cite{Nych} can be used to obtain approximate
realisations. An example on $W = [0,10] \times [0,20]$ with 
\[
m(x,y) = \frac{x+y}{10}
\]
and parameters $\sigma^2 = 1$, $\beta = 0.8$  viewed through independent 
Boolean models is depicted in Figure~\ref{F:Boolean}.

For a linked random field model, let $(X_1, X_2)$ consist of two 
independent stationary Boolean models with balls as primary grains, and 
set
\[
(\Psi_1, \Psi_2) =  
\left(
  \int_{X_1} e^{\Gamma_0(x)} dx, \int_{X_2} e^{\Gamma_0(x)} dx 
\right).
\]
Here, the common random field, although viewed through independent
spectres, causes positive association between the components of $\Psi$.

The estimated cross statistics are shown in Figure~\ref{F:Boolstat}
for $\Gamma_0$ as in Figure~\ref{F:Boolean} and Boolean models having
germ intensity $1/2$ and grain radius $r=1/2$.
The graph of $\widehat{L_{ij}(t)}$ lies below that of $\widehat{L_{j}(t)}$ 
reflecting the attraction between the components. The graph of
$\widehat{K_{ij}(t)}$ lies above that of the function $t\to \pi t^2$,
which confirms the positive correlation between the components.

An example of a random thinning field on $W = [0,10] \times [0,20]$ with 
\[
1 - r_2(x,y ) = r_1(x, y) =  \frac{y}{20}
\]
applied to $\exp[ \Gamma_0(\cdot)]$ with $\Gamma_0$ having mean zero 
and covariance function (\ref{e:cexp}) for $\sigma^2 = 1 $ and 
$\beta = 0.8$, and $X$ consisting of independent Boolean models as 
described above is shown in Figure~\ref{F:BoolThin}. 
Note that first component of the corresponding random measure $\Psi$ 
tends to place larger mass towards the top of $W$ (left panel), 
whereas the second components tends to place its mass near the bottom
(right panel of Figure~\ref{F:BoolThin}).

Although the first order structures -- as displayed in Figures~\ref{F:BoolThin}
and \ref{F:Boolean} -- of the random thinning field and the linked random
field model are completely different, their interaction structures
coincide and so do their cross statistics (cf.\ Figure~\ref{F:Boolstat}).

\section{Conclusion}

In this paper, we introduced summary statistics to quantify the 
correlation between the components of coverage-reweighted moment 
stationary multivariate random measures inspired by the $F$-, 
$G$- and $J$-statistics for point processes 
\cite{CronLies16,Lies11,LiesBadd99}.
The role of the generating functional in these papers is taken 
over by the Laplace functional and that of the product densities
by the coverage functions. Our statistics can also be seen as 
generalisations of the correlation measures introduced in 
\cite{StoyOhse82} for stationary random closed sets. 

To the best of our knowledge, such cross statistics for inhomogeneous
marked sets have not been proposed before. Under the strong
assumption of stationarity, however, some statistics were suggested.
Foxall and Baddeley \cite{FoxBadd02} defined a cross $J$-function 
for the dependence of a random closed set $X$ -- a line segment 
process in their geological application -- on a point pattern $Y$ by 
\[
J(t) = \frac{\PP^{0}(d((0, X) > t ) }{
\PP( d(0, X) > t )}
\]
where $\PP^0$ is the Palm distribution of $Y$, whereas 
Kleinschroth {\em et al.} \cite{Klei13} replaced the numerator by 
\[ 
\PP^{(0,i)}( \Psi_j(B(0,t)) = 0 )
\]
for the random length-measures $\Psi_j$ associated to a bivariate line 
segment process. It is not clear, though, how to generalise the resulting 
statistics to non-homogeneous models, as the moment measure of the random 
length-measure may not admit a Radon--Nikodym derivative.

\newpage

\section*{Figures}

\begin{figure}[htb]
\begin{center}
\centerline{
\epsfxsize=0.4\hsize
\epsffile{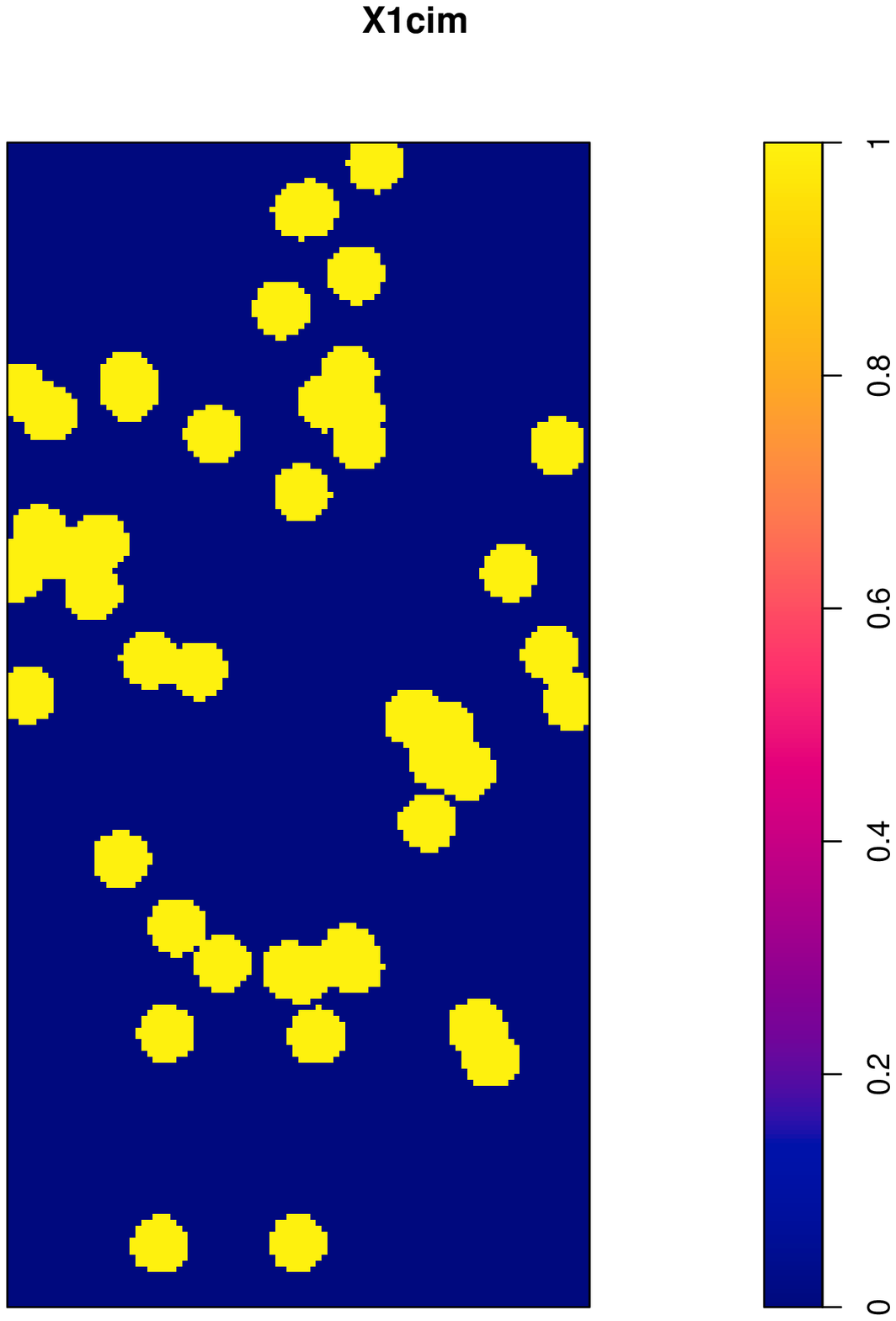}
\hspace{0.5cm}
\epsfxsize=0.4\hsize
\epsffile{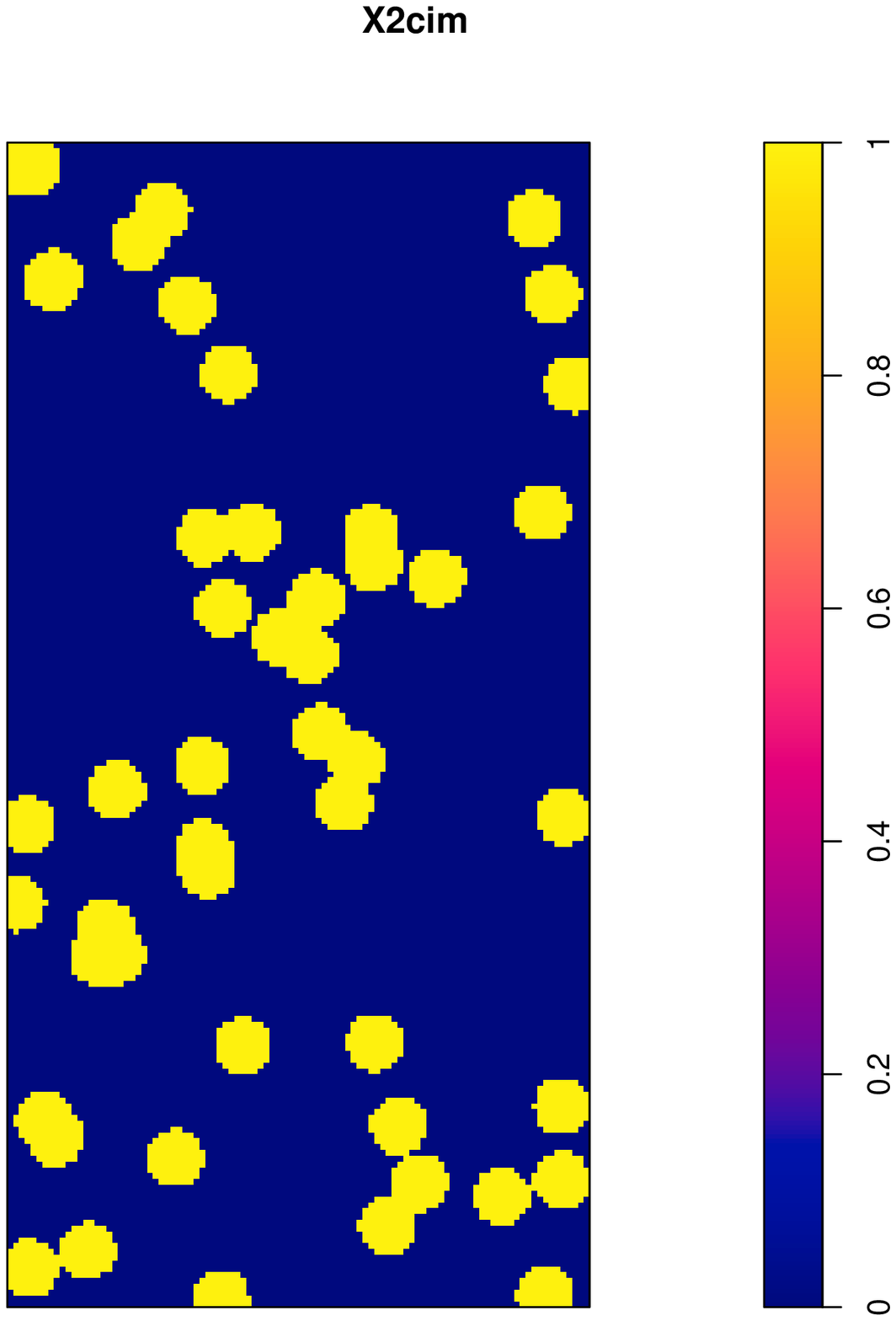}}
\end{center}
\caption{Images of $U_{r/2}(\phi_1)$ (left) and $U_{r/2}(\phi_2)$
for a realisation $(\phi_1, \phi_2)$ of the Widom--Rowlinson mixture 
model with $\beta_1 = \beta_2 = 1$ on $W=[0,10]\times [0,20]$ for
$r=1$.}
\label{F:WR}
\end{figure}

\begin{figure}[htb]
\begin{center}
\centerline{
\epsfxsize=0.3\hsize
\epsffile{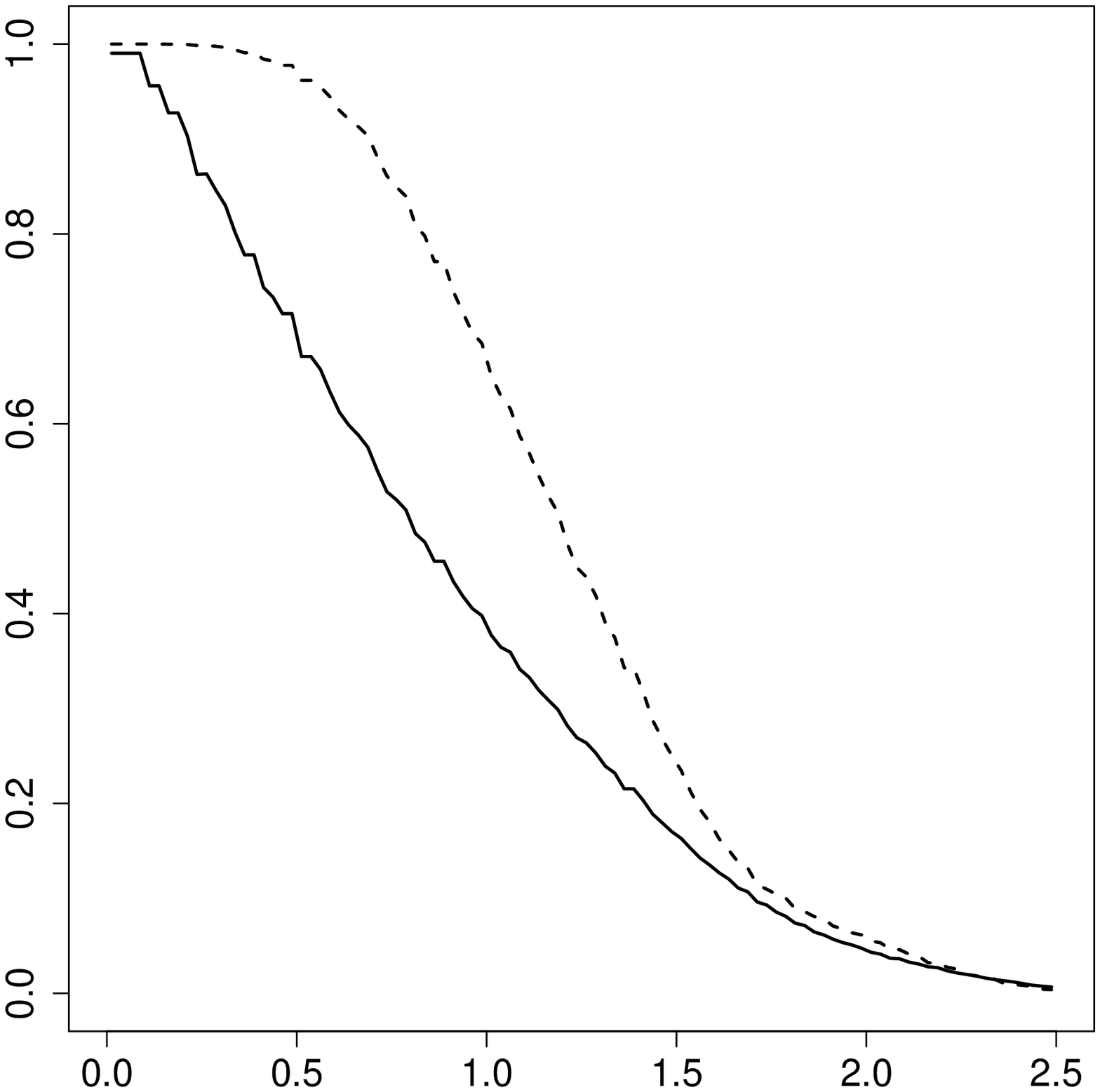}
\hspace{0.5cm}
\epsfxsize=0.3\hsize
\epsffile{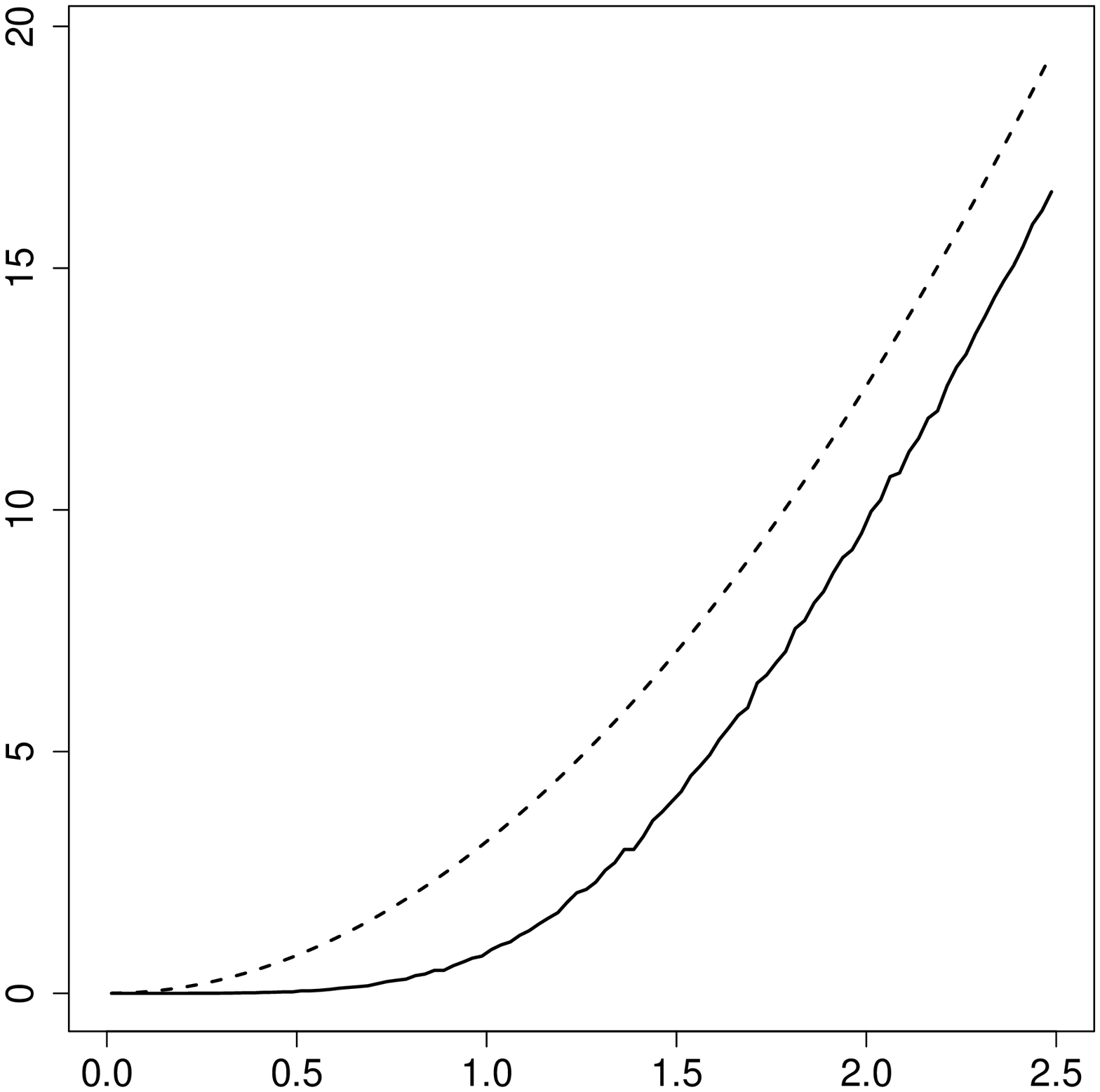}}

\centerline{
\epsfxsize=0.3\hsize
\epsffile{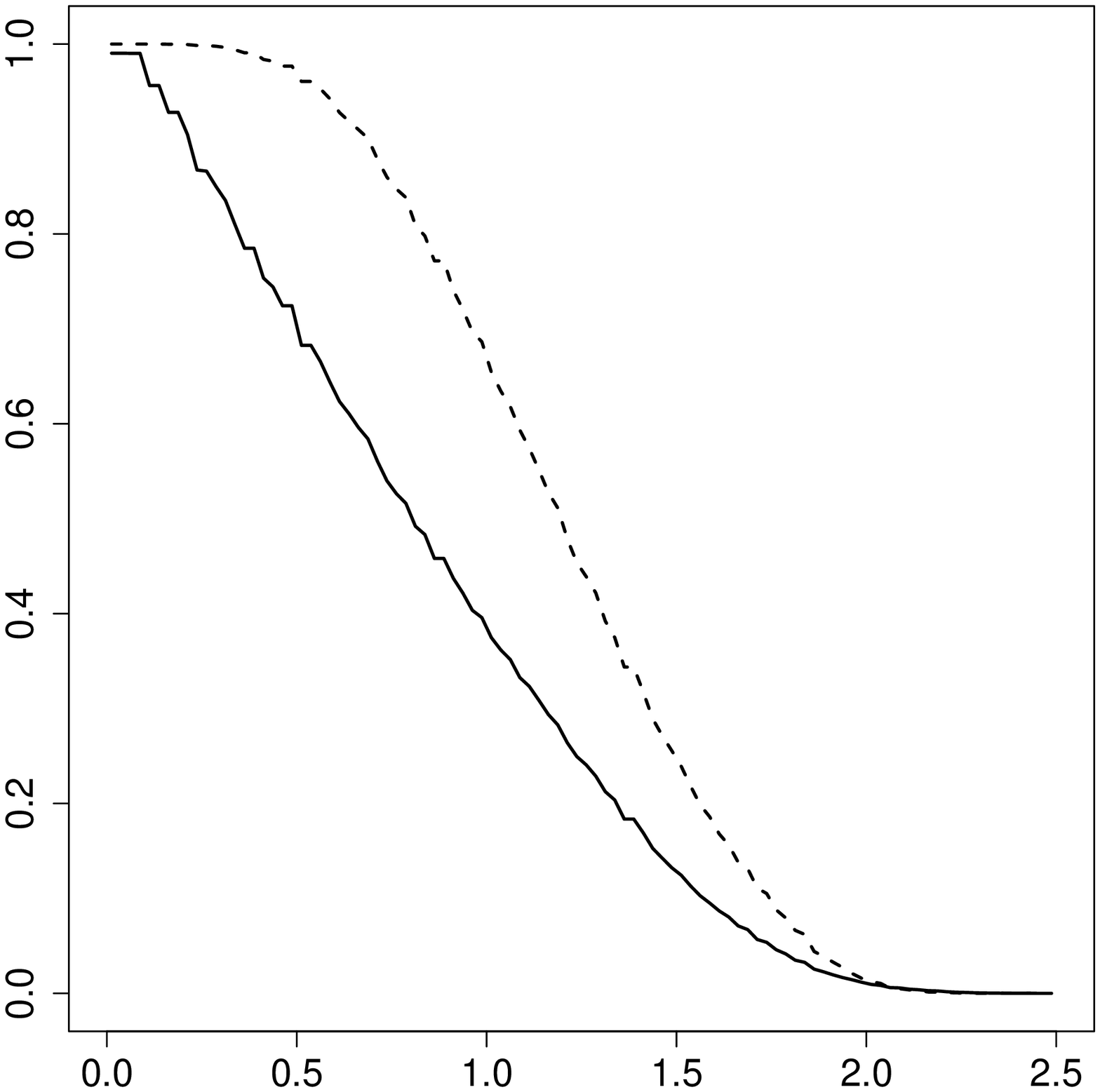}
\hspace{0.5cm}
\epsfxsize=0.3\hsize
\epsffile{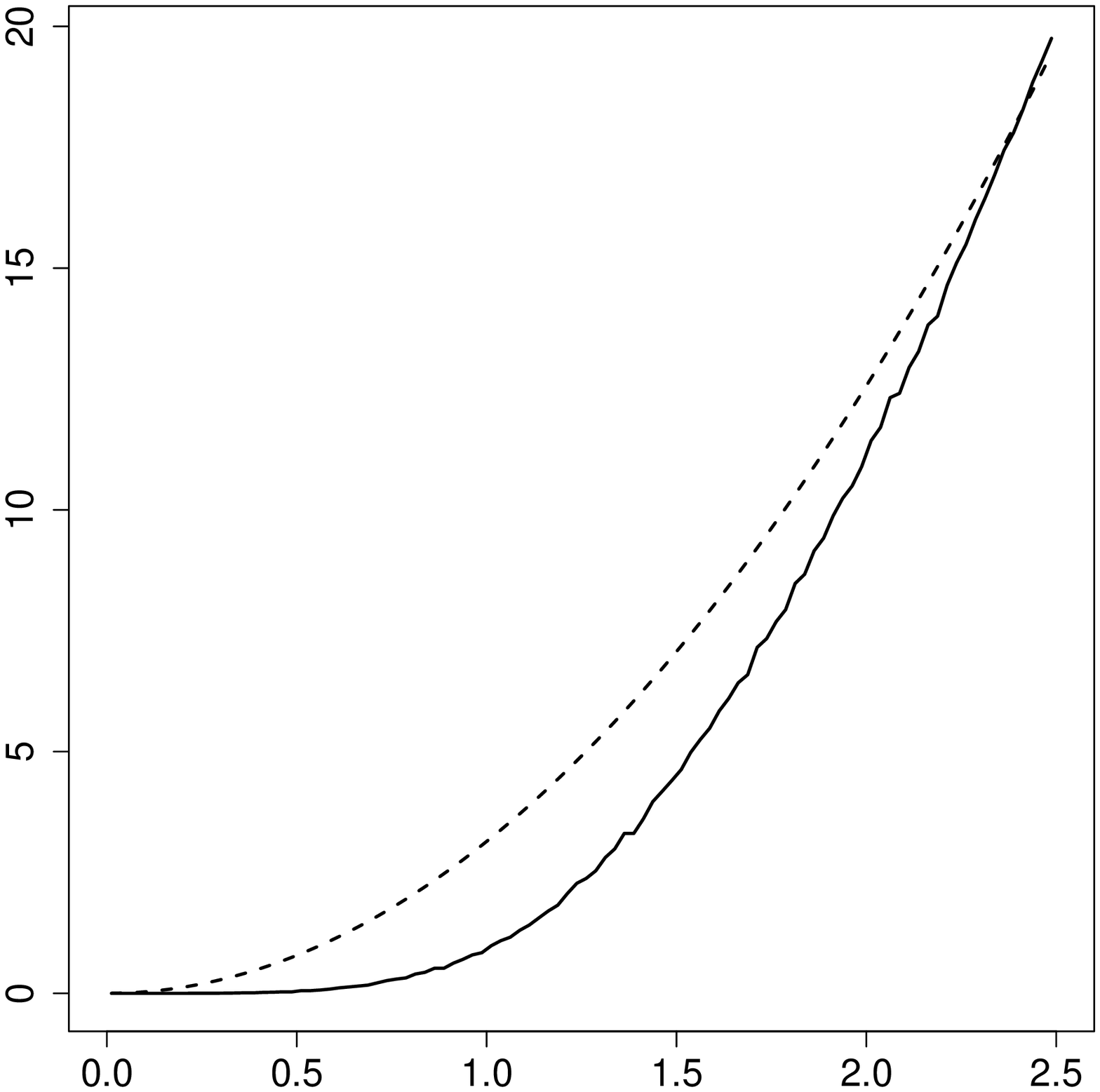}}
\end{center}
\caption{Estimated cross statistics for the data of Figure~\ref{F:WR}.
Top row: $\widehat{L_2(t)}$ (solid) and $\widehat{L_{12}(t)}$ (dotted)
plotted against $t$ (left); $\widehat{K_{12}(t)}$ (solid) and $\pi t^2$ 
(dotted) plotted against $t$ (right).
Bottom row: $\widehat{L_1(t)}$ (solid) and $\widehat{L_{21}(t)}$ (dotted)
plotted against $t$ (left); $\widehat{K_{21}(t)}$ (solid) and $\pi t^2$ 
(dotted) plotted against $t$ (right).}
\label{F:WRstat}
\end{figure}

\begin{figure}[htb]
\begin{center}
\centerline{
\epsfxsize=0.4\hsize
\epsffile{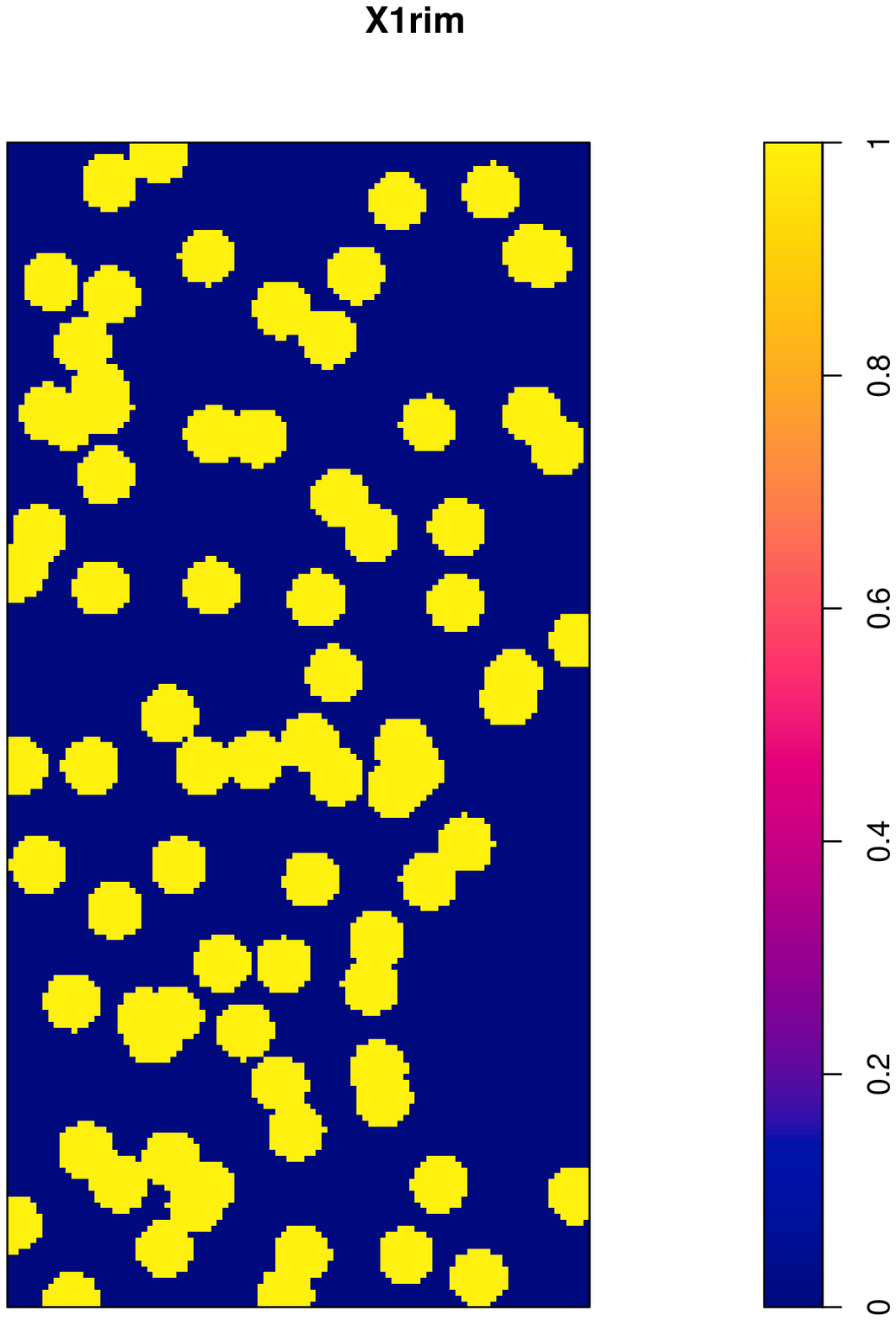}
\hspace{0.5cm}
\epsfxsize=0.4\hsize
\epsffile{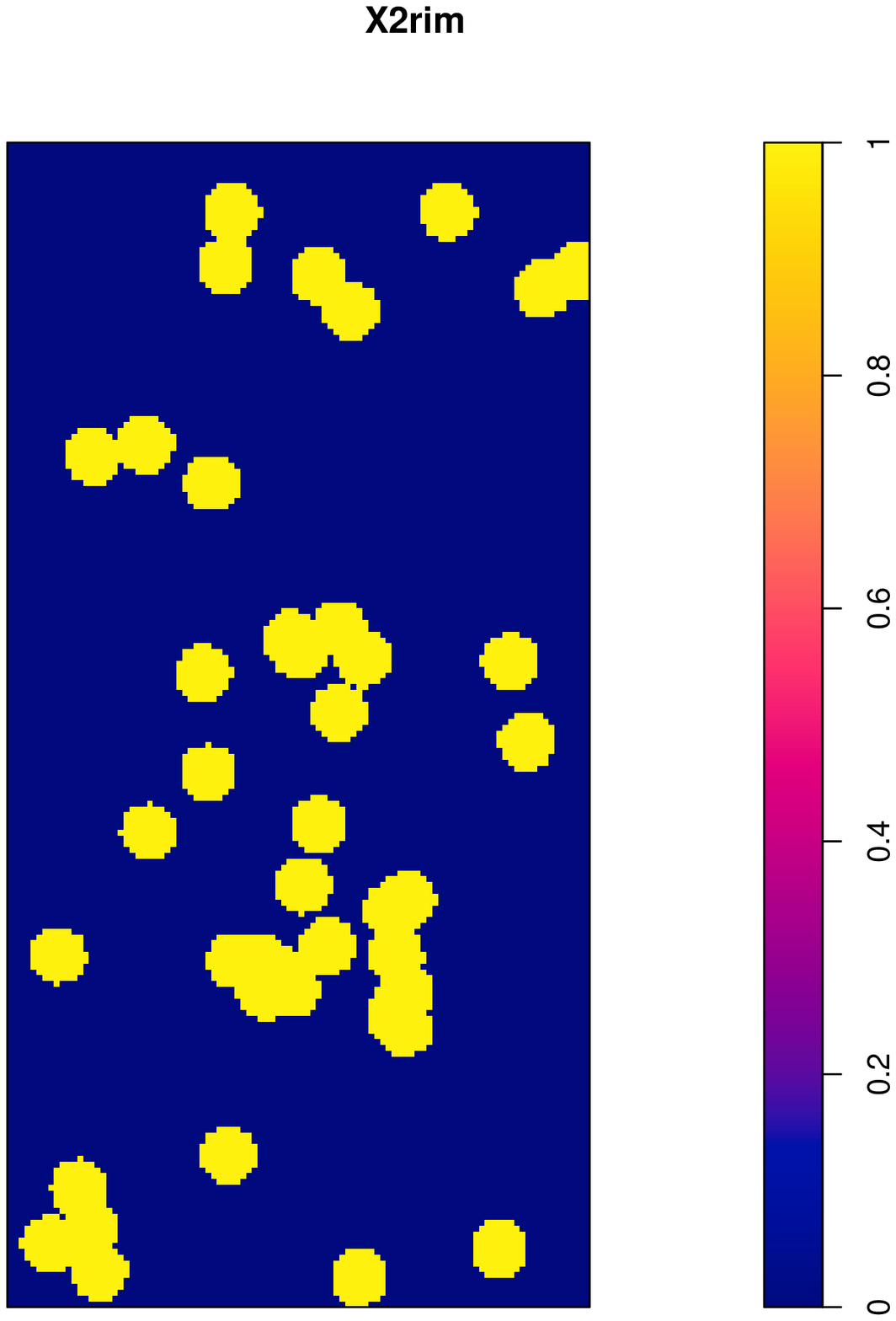}}
\end{center}
\caption{Images of $U_{r/2}(\phi_1)$ (left) and $U_{r/2}(\phi_2)$
for a realisation $(\phi_1, \phi_2)$ of the dual Widom--Rowlinson mixture 
model with $\beta_1 = \beta_2 = 1/4$ on $W=[0,10]\times [0,20]$ for
$r=1$.}
\label{F:dualWR}
\end{figure}

\begin{figure}[htb]
\begin{center}
\centerline{
\epsfxsize=0.3\hsize
\epsffile{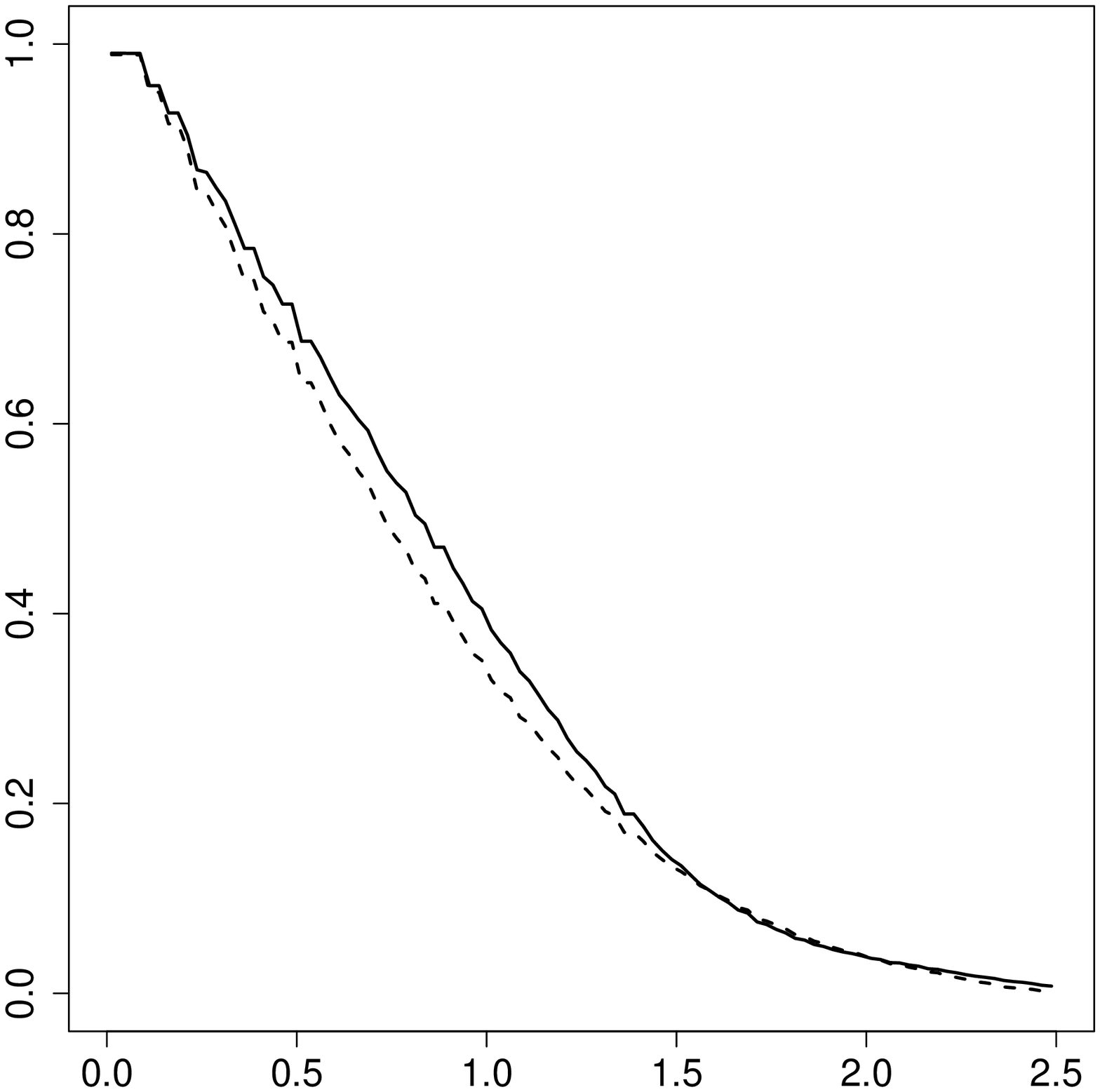}
\hspace{0.5cm}
\epsfxsize=0.3\hsize
\epsffile{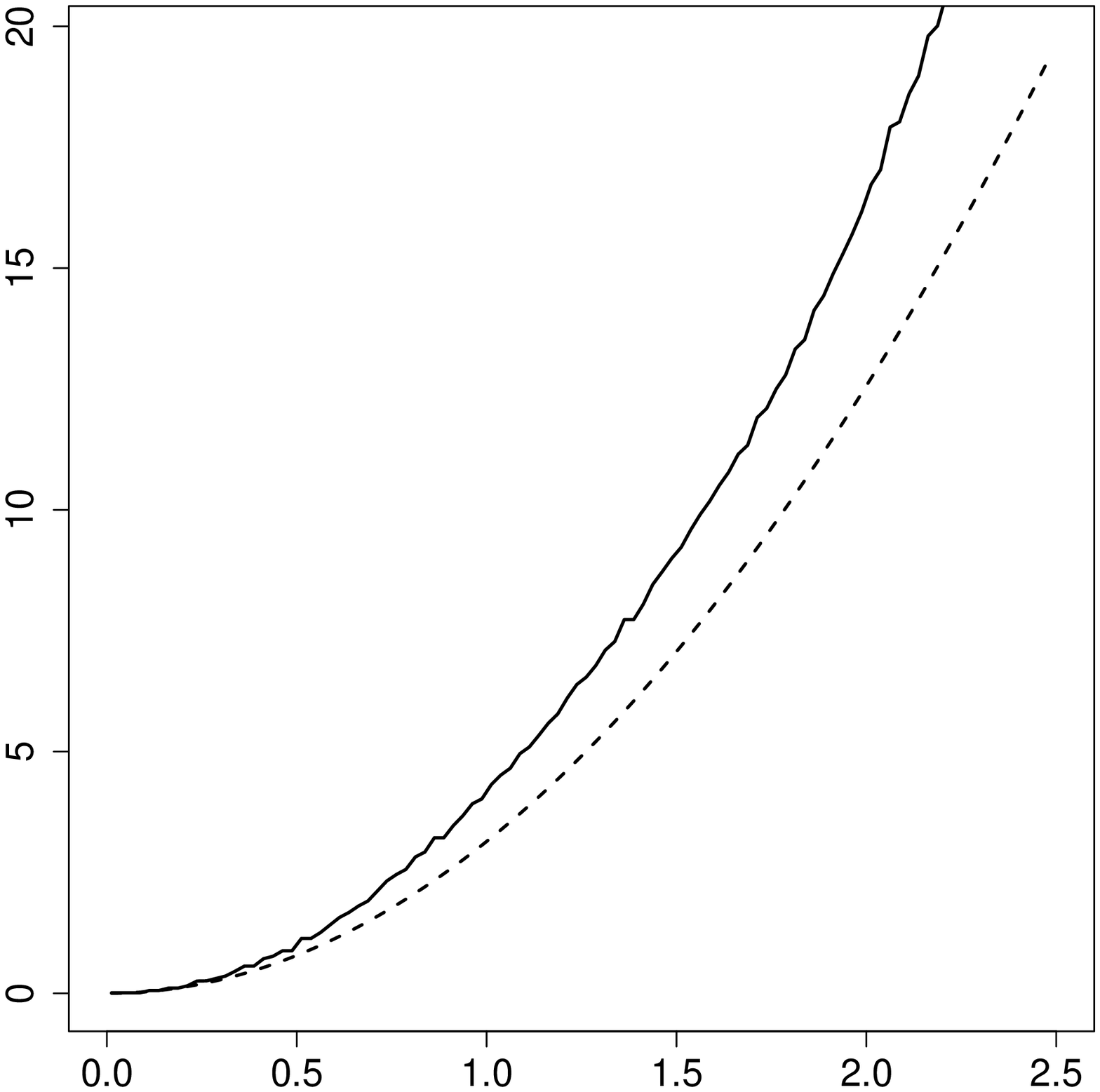}}

\

\centerline{
\epsfxsize=0.3\hsize
\epsffile{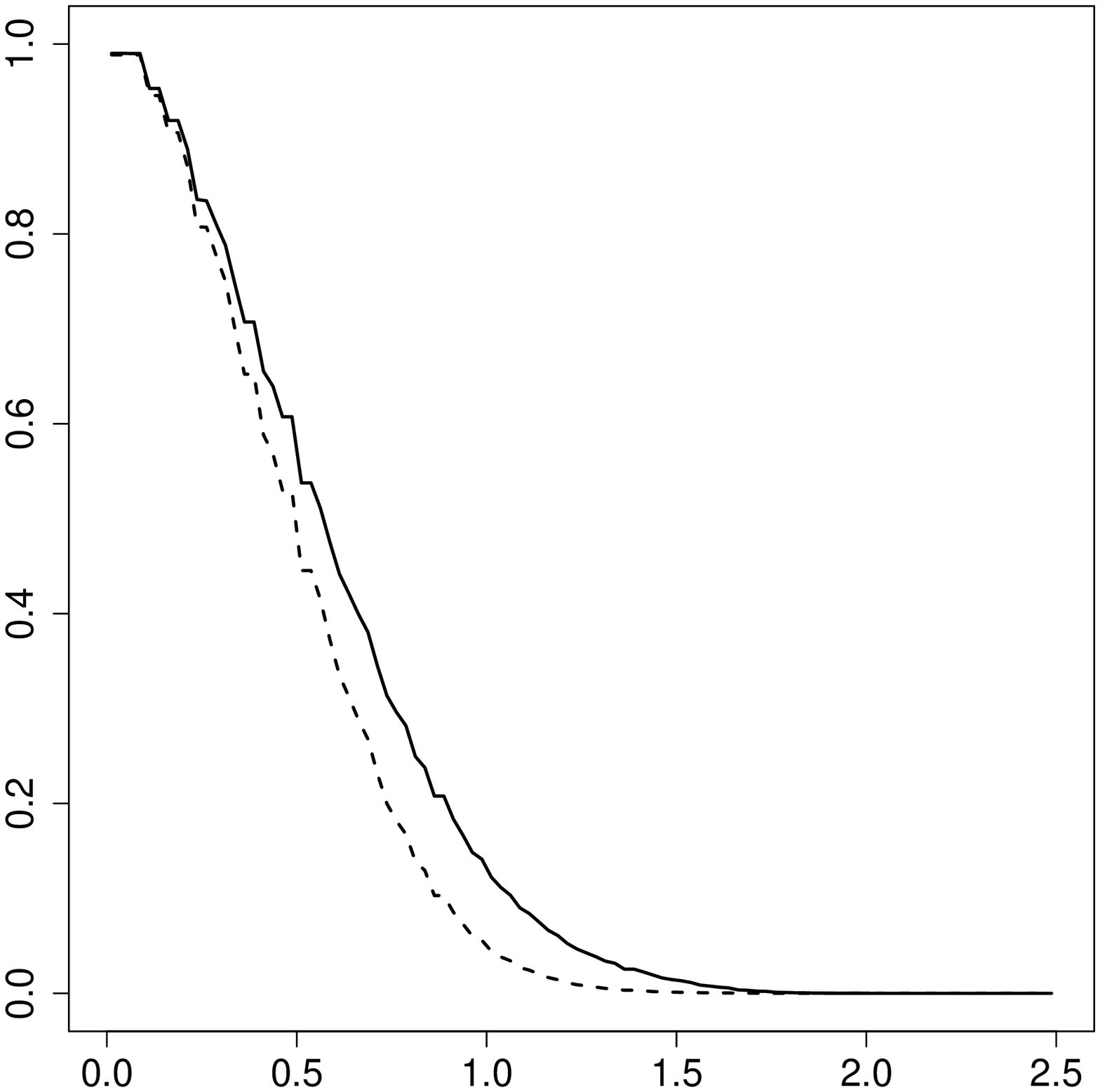}
\hspace{0.5cm}
\epsfxsize=0.3\hsize
\epsffile{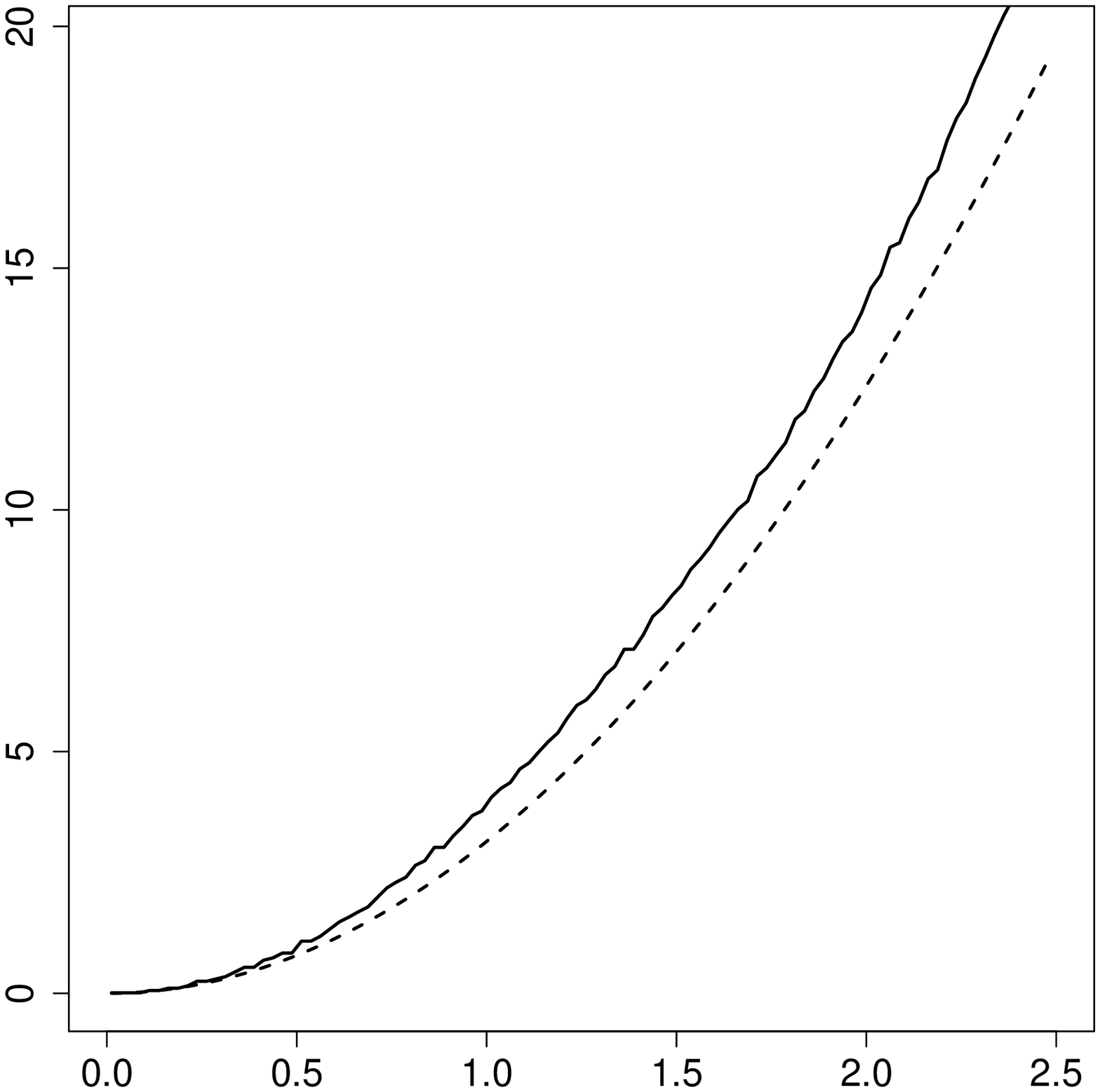}}
\end{center}
\caption{Estimated cross statistics for the data of Figure~\ref{F:dualWR}.
Top row: $\widehat{L_2(t)}$ (solid) and $\widehat{L_{12}(t)}$ (dotted)
plotted against $t$ (left); $\widehat{K_{12}(t)}$ (solid) and $\pi t^2$
(dotted) plotted against $t$ (right).
Bottom row: $\widehat{L_1(t)}$ (solid) and $\widehat{L_{21}(t)}$ (dotted)
plotted against $t$ (left); $\widehat{K_{21}(t)}$ (solid) and $\pi t^2$ 
(dotted) plotted against $t$ (right).}
\label{F:dualstat}
\end{figure}

\begin{figure}[htb]
\begin{center}
\centerline{
\epsfxsize=0.4\hsize
\epsffile{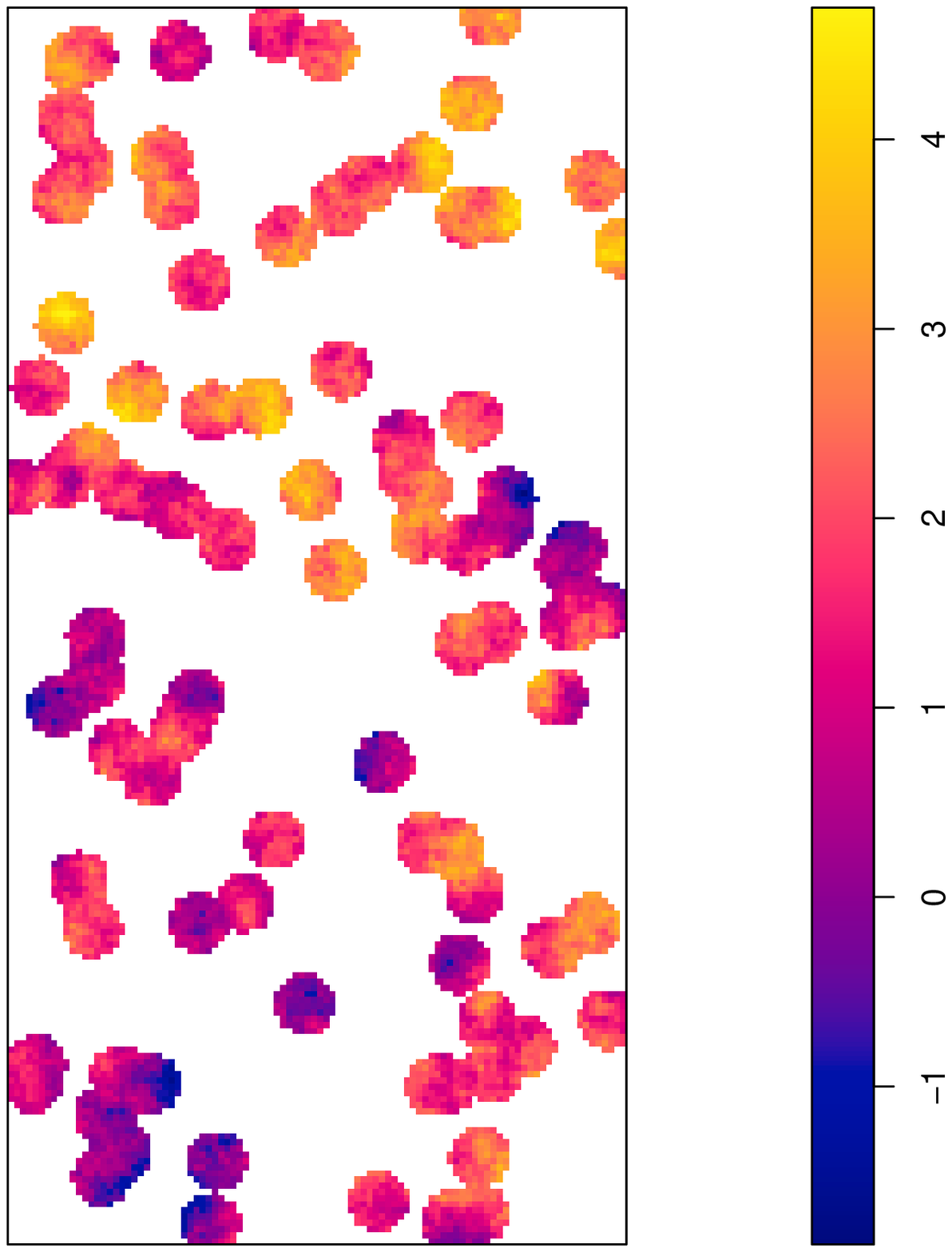}
\hspace{0.5cm}
\epsfxsize=0.4\hsize
\epsffile{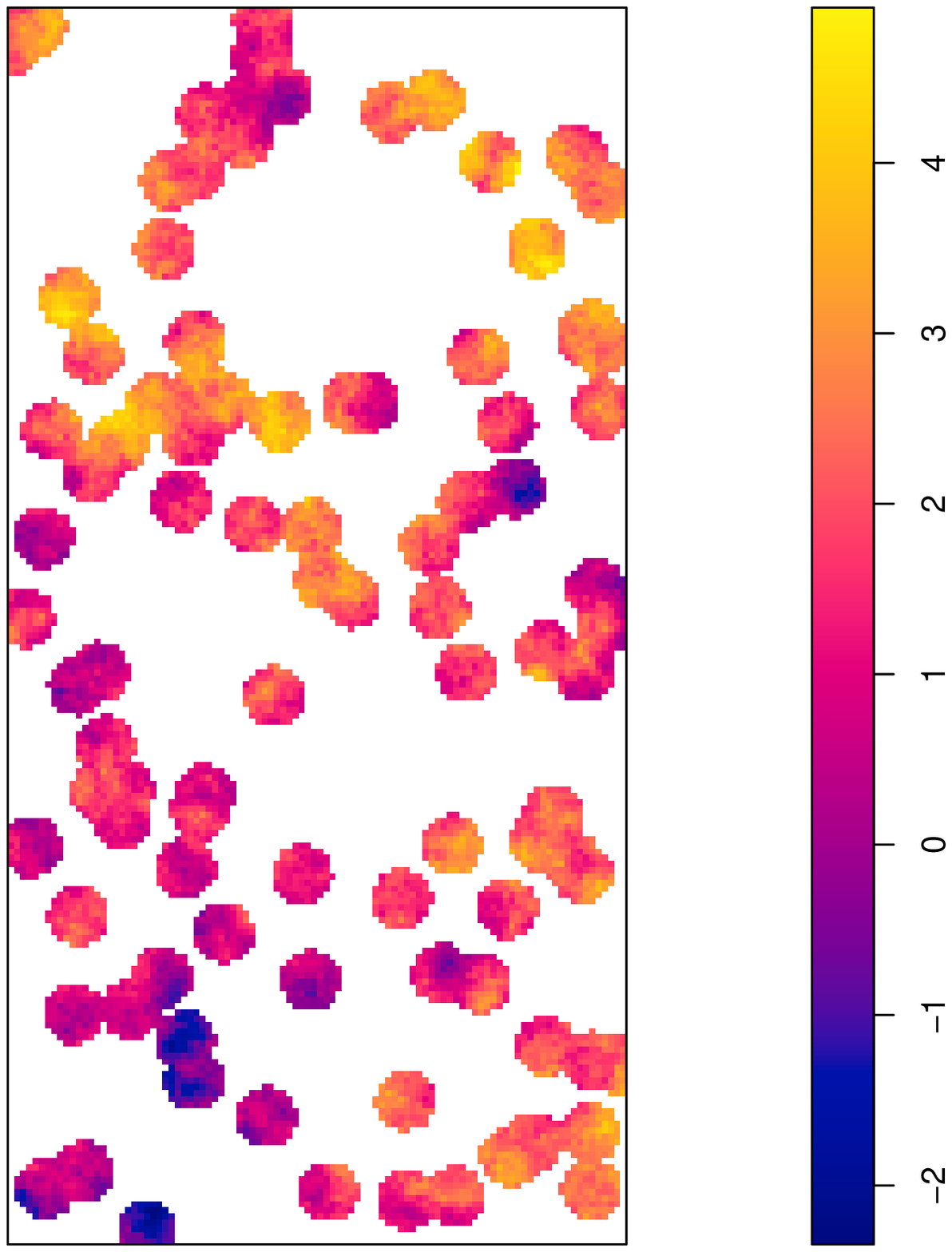}}
\end{center}
\caption{Images of a Gaussian random field  on $W=[0,10]\times [0,20]$ 
with mean function $m(x,y) = (x+y)/10$ and covariance function 
$\sigma^2 \exp[ - \beta || \cdot || ]$ for $\beta = 0.8$ and 
$\sigma^2 = 1$  viewed through independent Boolean models 
$X_1$ (left) and $X_2$ (right) with germ intensity $1/2$ and 
primary grain radius $1/2$. }
\label{F:Boolean}
\end{figure}

\begin{figure}[htb]
\begin{center}
\centerline{
\epsfxsize=0.3\hsize
\epsffile{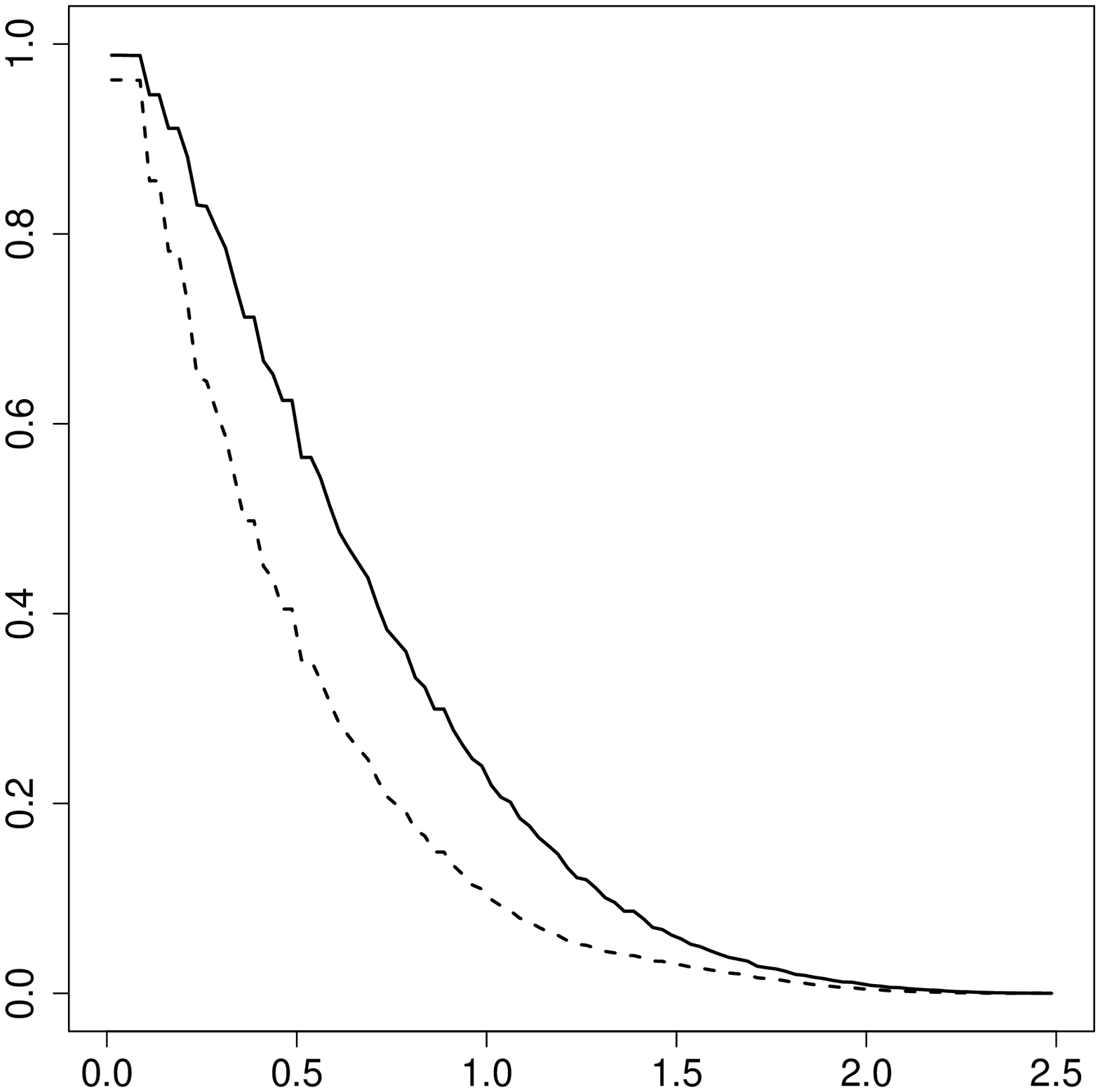}
\hspace{0.5cm}
\epsfxsize=0.3\hsize
\epsffile{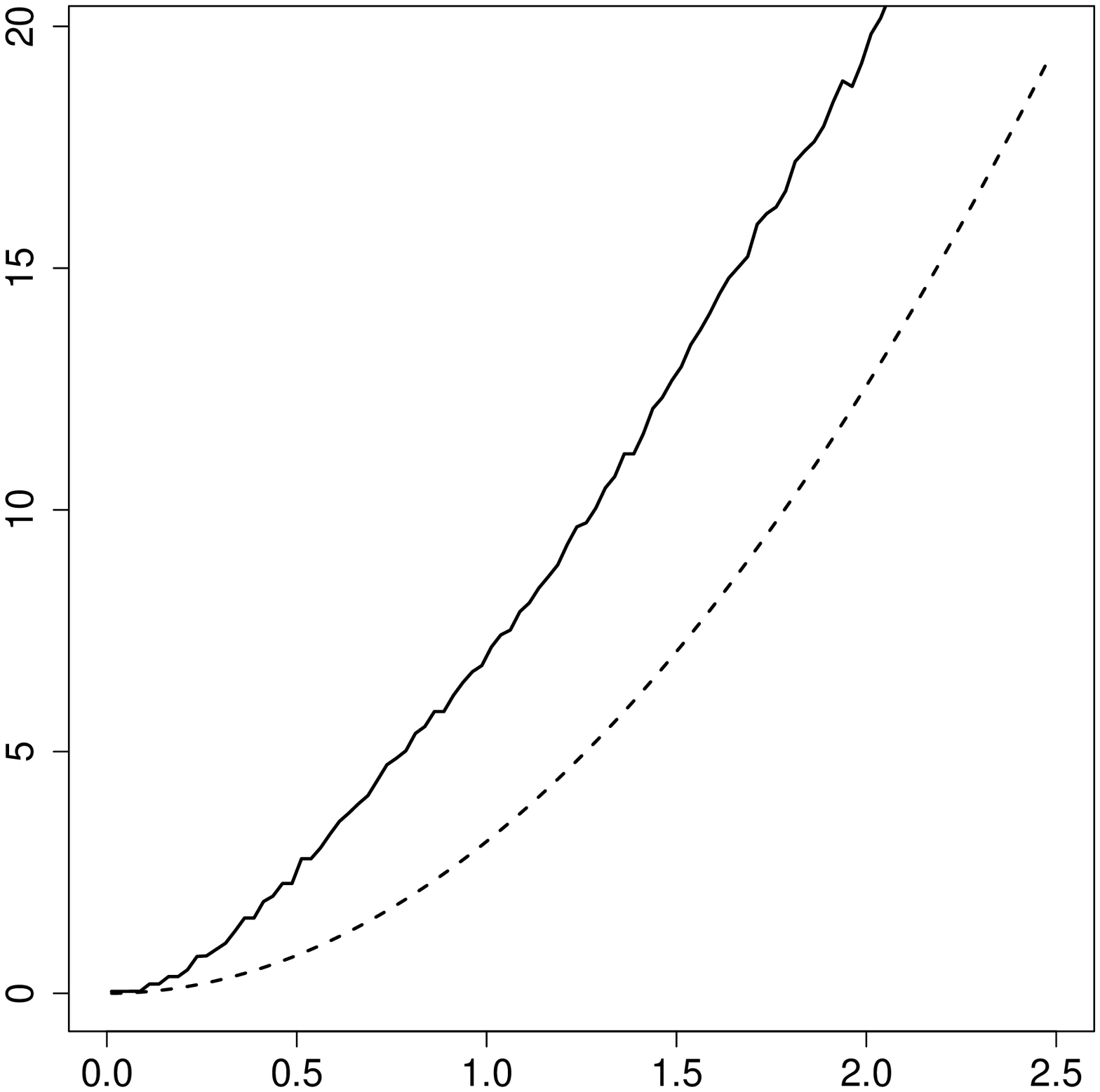}}

\centerline{
\epsfxsize=0.3\hsize
\epsffile{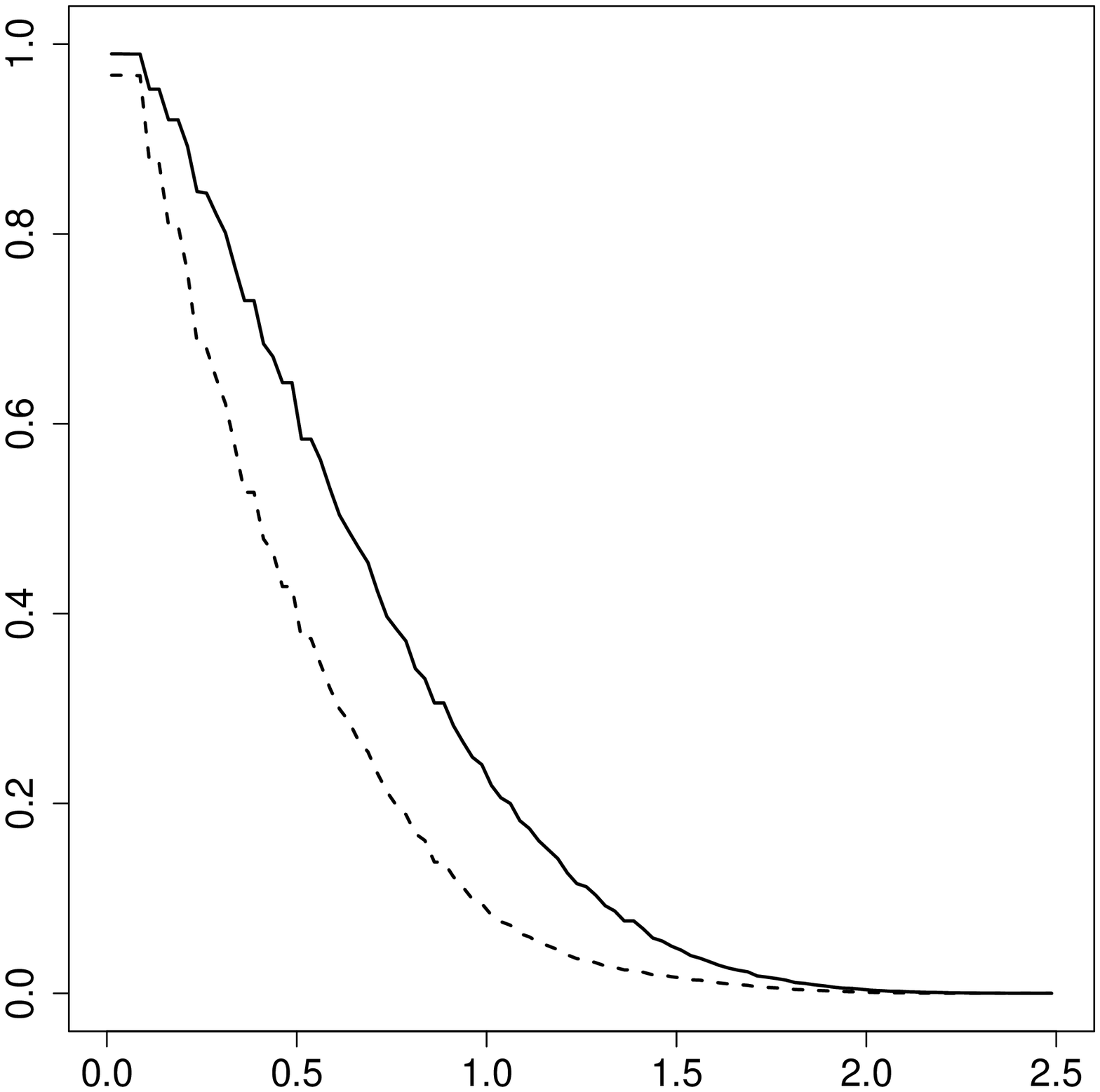}
\hspace{0.5cm}
\epsfxsize=0.3\hsize
\epsffile{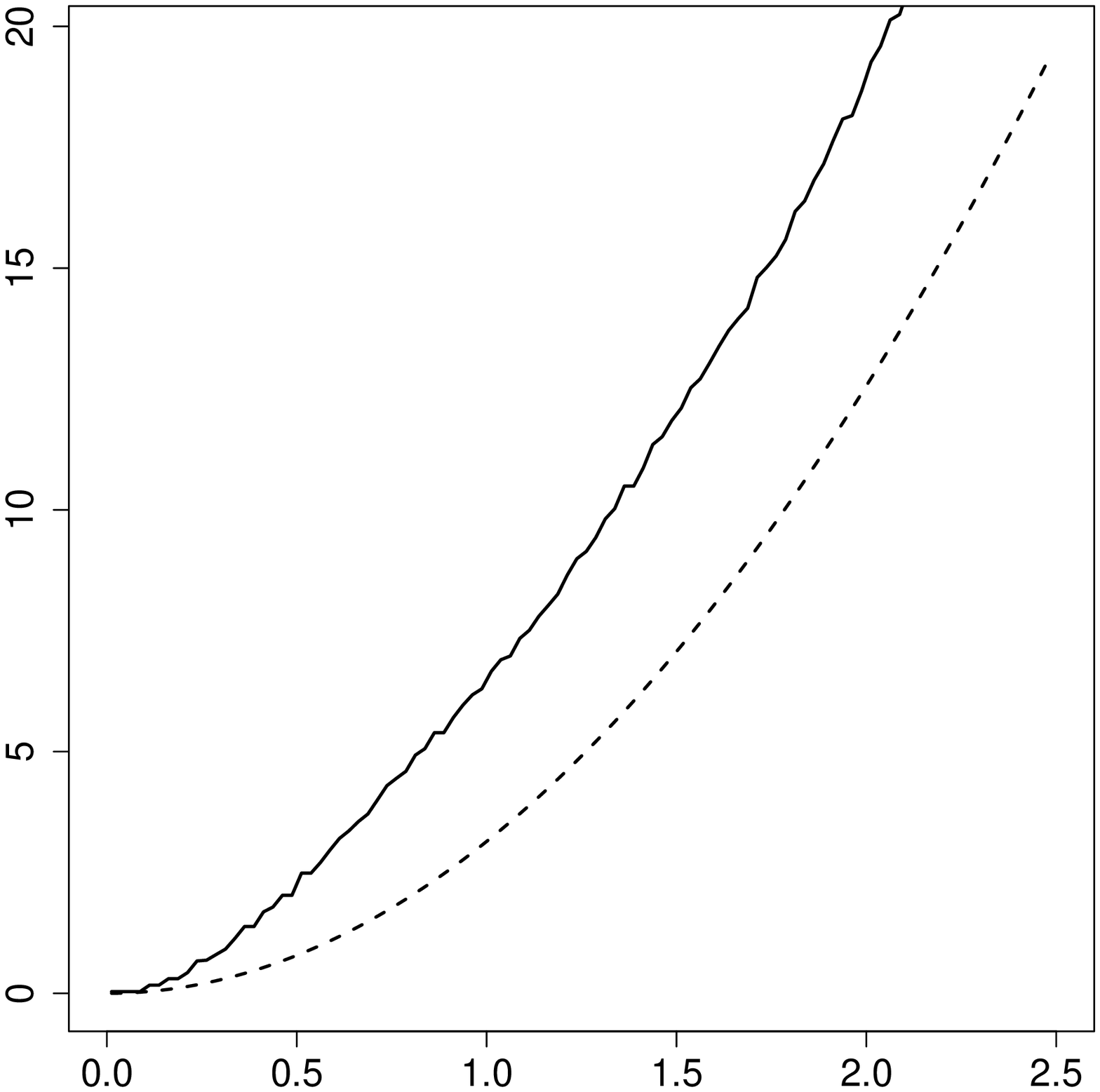}}
\end{center}
\caption{Estimated cross statistics for a random field model 
on $W=[0,10]\times [0,20]$ defined by $X$ and $\Gamma$ as follows:
$\Gamma_1(x) = \Gamma_2(x) = \exp[ Z(x) ]$ where $Z$ is a Gaussian 
random field with mean function $m(x,y) = (x+y)/10$ and covariance function
$\sigma^2 \exp[ - \beta || \cdot || ]$ for $\beta = 0.8$ and
$\sigma^2 = 1$; the components of $X$ are independent Boolean models
with germ intensity $1/2$ and primary grain radius $1/2$,  
cf.\ Figure~\ref{F:Boolean}.
Top row: $\widehat{L_2(t)}$ (solid) and $\widehat{L_{12}(t)}$ (dotted)
plotted against $t$ (left); $\widehat{K_{12}(t)}$ (solid) and $\pi t^2$
(dotted) plotted against $t$ (right).
Bottom row: $\widehat{L_1(t)}$ (solid) and $\widehat{L_{21}(t)}$ (dotted)
plotted against $t$ (left); $\widehat{K_{21}(t)}$ (solid) and $\pi t^2$ 
(dotted) plotted against $t$ (right).}
\label{F:Boolstat}
\end{figure}

\begin{figure}[htb]
\begin{center}
\centerline{
\epsfxsize=0.4\hsize
\epsffile{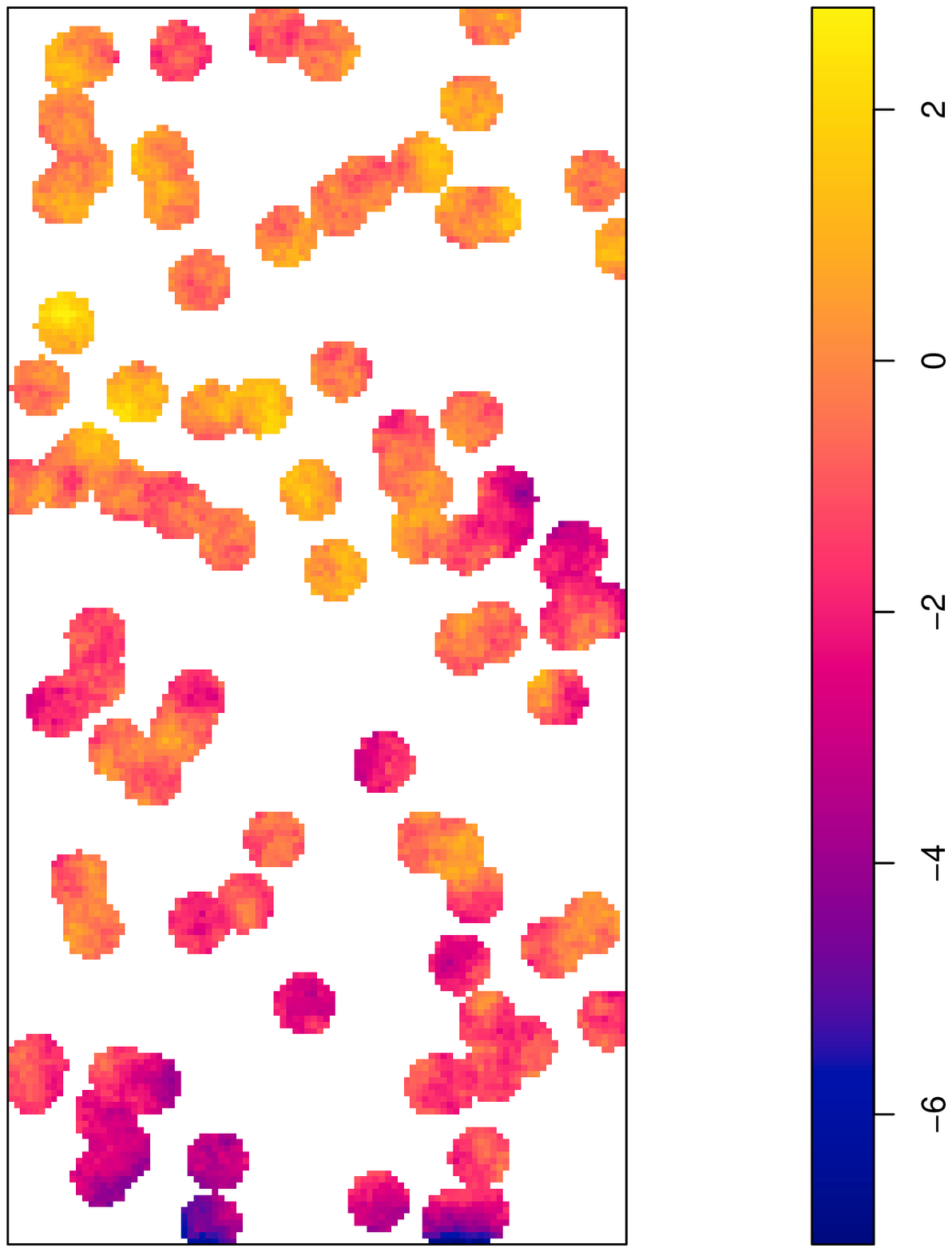}
\hspace{0.4cm}
\epsfxsize=0.4\hsize
\epsffile{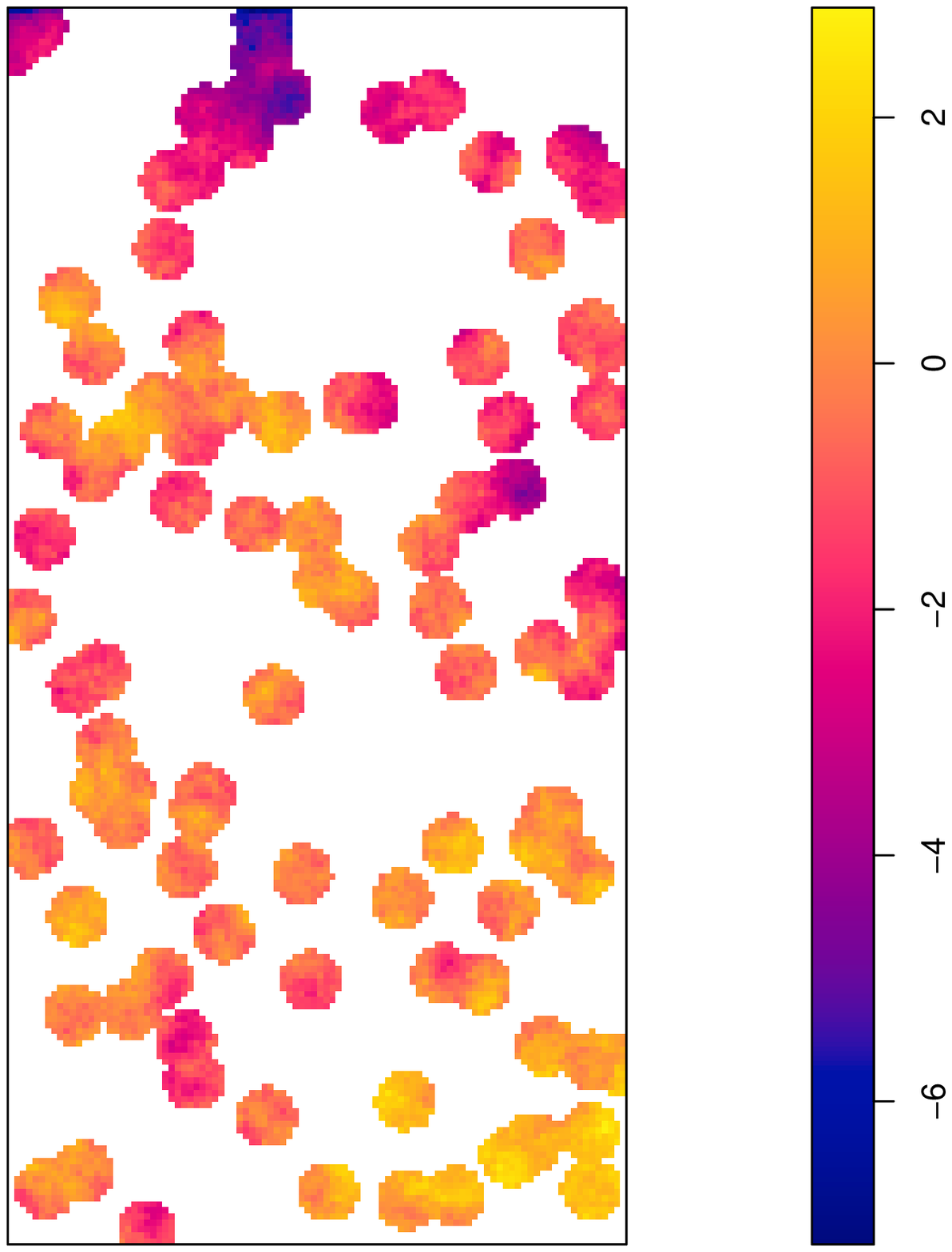}}
\end{center}
\caption{Images of realisations $(\log \psi_1, \log \psi_2)$ of 
a random thinning field model $(\Psi_1, \Psi_2)$ on $W=[0,10]\times [0,20]$ 
with $r_1(x,y) = y/20$, $\log \Gamma_0$ a mean zero Gaussian random field
with covariance function $\sigma^2 \exp[ - \beta || \cdot || ]$ for 
$\beta = 0.8$ and $\sigma^2 = 1$, and $X$ as in Figure~\ref{F:Boolstat}.}
\label{F:BoolThin}
\end{figure}

\clearpage
\section*{Appendix}

\begin{res}
\label{l:old}
Let $\alpha$ be an $\R^+$-valued random variable with finite 
second moment. Then
\[
 \left( \EE\left[ \alpha \exp\left[-\alpha t^d \right]\right] \right)^2 
\leq 
 \EE\left[ \alpha^2 \exp\left[-\alpha t^d\right] \right] \;
 \EE \exp\left[-\alpha t^d\right].
\]
\end{res}

\begin{proof}
Let $X$ and $Y$ be independent random variables with the
same distribution as $\alpha$. Then we need to show that
\begin{equation}
\label{E:mp}
 \EE\left[ X(Y-X) e^{-(X+Y) t^d} \right] \leq 0.
\end{equation}
The expectation (\ref{E:mp}) equals
\[
 \EE\left[ X(Y-X) e^{-(X+Y)t^d} 1\{Y>X\} \right]  +
 \EE\left[ X(Y-X) e^{-(X+Y)t^d} 1\{X>Y \} \right] + 0 
\]
\[
  =  
 \EE\left[ X(Y-X) e^{-(X+Y) t^d} 1\{ Y>X \} \right]  +
 \EE\left[ Y(X-Y) e^{-(Y+X) t^d} 1 \{Y>X\} \right]  
\]
\[
  =  
 \EE\left[ -(X^2+Y^2-2XY) e^{-(X+Y) t^d} 1\{Y>X\} \right]  
\]
\[
  =  
 - \EE\left[ (Y-X)^2 e^{-(X+Y) t^d} 1\{Y>X\} \right]  \leq 0.
\]
\end{proof}

For compound measures, by Theorem~\ref{t:compound},
the derivative of $J_{12}$ is 
\[
  J_{12}'(t) =  d \kappa_d t^{d-1}  J_{12}(t) \left[
     \frac{ \EE\left[ \Lambda_2 e^{-\Lambda_2 \kappa_d t^d / \EE\Lambda_2} \right]}{
         \EE\Lambda_2 \,   \EE\left[ e^{-\Lambda_2 \kappa_d t^d / \EE\Lambda_2} \right]} 
     -  
     \frac{ \EE\left[ \Lambda_1 \Lambda_2  e^{-\Lambda_2 \kappa_d t^d / \EE\Lambda_2} 
    \right]}{  \EE\Lambda_2 \, \EE\left[ \Lambda_1
       e^{-\Lambda_2 \kappa_d t^d / \EE\Lambda_2} \right]}  \right] .
\]

Upon plugging in the model assumptions and applying Lemma~\ref{l:old},
one derives that for the linked model, $J_{12}(t)$ is monotonically
non-increasing; in the balanced case, $J_{12}(t)$ is non-decreasing. 
For example for the balanced case, 
\[
   \EE\left[ \Lambda_2 e^{-\Lambda_2 \kappa_d t^d / \EE\Lambda_2} \right]
 \EE\left[ \Lambda_1  e^{-\Lambda_2 \kappa_d t^d / \EE\Lambda_2} \right]
- 
  \EE\left[ e^{-\Lambda_2 \kappa_d t^d / \EE\Lambda_2} \right]
 \EE\left[ \Lambda_1 \Lambda_2  e^{-\Lambda_2 \kappa_d t^d / \EE\Lambda_2} \right] 
\]
\[
= \EE\left[ \Lambda_2^2 e^{-\Lambda_2 \kappa_d t^d / \EE\Lambda_2} \right]
\EE\left[ e^{-\Lambda_2 \kappa_d t^d / \EE\Lambda_2} \right] 
-
\left(\EE\left[ \Lambda_2 e^{-\Lambda_2 \kappa_d t^d / \EE\Lambda_2} \right] \right)^2 
\]
and setting $\alpha = \Lambda_2 \kappa_d/\EE\Lambda_2$ yields the
desired result.


\clearpage

\begin{thebibliography}{99}

\bibitem{Adle81}
Adler, R.J. (1981).
{\em The geometry of random fields.}
John Wiley \& Sons.

\bibitem{AyalSimo98}
Ayala, G., Sim\'o, A. (1998).
Stochastic labelling of biological images.
{\em Statistica Neerlandica\/} 52:141--152.

\bibitem{Ayal91}
Ayala, G., Ferrandiz, J., Montes, F. (1991).
Random set and coverage measure.
{\em Advances in Applied Probability\/} 23:972--974.

\bibitem{Badd00}
Baddeley, A.J., M{\o}ller, J., Waagepetersen, R. (2000). 
Non- and semi-parametric estimation of interaction in inhomogeneous 
point patterns. 
{\em Statistica Neerlandica\/} 54:329--350.

\bibitem{Ball12}
Ballani, F., Kabluchenko, Z., Schlather, M. (2012).
Random marked sets.
{\em Advances in Applied Probability\/} 44:603--616.

\bibitem{Chiu13}
Chiu, S.N., Stoyan, D., Kendall, W.S., Mecke, J. (2013).
{\em Stochastic geometry and its applications\/}.
Third edition. John Wiley \& Sons. 

\bibitem{Coeu15}
Coeurjolly, J.-F., M\o ller, J., Waagepetersen, R. (2015).
Palm distributions for log Gaussian Cox processes.
Arxiv 1506.04576, June 2015.

\bibitem{CronLies16}
Cronie, O., Lieshout, M.N.M.~van (2016). 
Summary statistics for inhomogeneous marked point processes.
{\em Annals of the Institute of Statistical Mathematics\/} 68:905--928.

\bibitem{DaleVere03}
Daley, D.J., Vere-Jones, D. (2003).
{\em An introduction to the theory of point processes: 
Volume I: Elementary theory and methods\/}.
Second edition. Springer.

\bibitem{DaleVere08}
Daley, D.J., Vere-Jones, D. (2008).
{\em An introduction to the theory of point processes: 
Volume II: General theory and structure\/}.
Second edition. Springer.

\bibitem{Digg83}
Diggle, P.J. (1983).
{\em Statistical analysis of spatial point patterns\/}.
First edition. Third edition 2014. CRC Press.

\bibitem{Esar67}
Esary, J.D., Proschan, F., Walkup, D.W. (1967).
Association of random variables with applications.
{\em The Annals of Mathematical Statistics\/} 38:1466--1474.

\bibitem{FoxBadd02}
Foxall, R., Baddeley, A. (2002).
Nonparametric measures of association between a spatial point process
and a random set, with geological applications.
{\em Applied Statistics\/} 51:165--182.

\bibitem{Gall14}
Gallego, M.A., Ib\'a\~nez, M.V., Sim\'o, A. (2014).
Inhomogeneous $K$-functions for germ-grain models.
Arxiv 1401.8115, January 2014.

\bibitem{Hagg99}
H\"{a}ggstr\"{o}m, O., Lieshout, M.N.M.~van, M\o ller, J. (1999).
Characterisation results and {Markov} chain {Monte Carlo}
algorithms including exact simulation for some spatial point processes.
{\em Bernoulli\/} 5:641--658.

\bibitem{KendMoll00}
Kendall, W.S., M\o ller, J. (2000).
Perfect simulation using dominating processes on ordered spaces, 
with application to locally stable point processes.
{\em Advances in Applied Probability (SGSA)} 32:844--865.

\bibitem{Koub16}
Koubek, A., Pawlas, Z., Brereton, T., Kriesche, B., Schmidt, V. (2016).
Testing the random field model hypothesis for random marked closed sets.
{\em Spatial Statistics\/} 16:118--136.

\bibitem{Kuma83}
Kumar, J.-D., Proschan, F. (1983).
Negative association of random variables with applications.
{\em The Annals of Statistics\/} 11:286--295.

\bibitem{Klei13}
Kleinschroth, F., Lieshout, M.N.M.~van, Mortier, F., Stoica, R.S. (2013).
Personal communication.

\bibitem{Lehm66}
Lehmann, E.L. (1966).
Some concepts of dependence.
{\em The Annals of Mathematical Statistics\/} 37:1137--1153.

\bibitem{Lies11}
Lieshout, M.N.M.~van (2011). 
A $J$-function for inhomogeneous point processes.
{\em Statistica Neerlandica\/} 65:183--201. 

\bibitem{LiesBadd96}
Lieshout, M.N.M.~van, Baddeley, A.J. (1996). 
A nonparametric measure of spatial interaction in point patterns.
{\em Statistica Neerlandica\/} 50:344--361. 

\bibitem{LiesBadd99}
Lieshout, M.N.M.~van, Baddeley, A.J. (1999). 
Indices of dependence between types in multivariate point patterns.
{\em Scandinavian Journal of Statistics\/} 26:511--532.

\bibitem{LiesStoi06}
Lieshout, M.N.M.~van, Stoica, R.S. (2006).
 Perfect simulation for marked point processes.
{\em Computational Statistics \& Data Analysis\/} 51:679--698.

\bibitem{Math75}
Matheron, G. (1975).
{\em Random sets and integral geometry\/}.
John Wiley \& Sons.

\bibitem{Molc97}
Molchanov, I.S. (1997).
{\em Statistics of the Boolean model for practitioners and mathematicians\/}.
John Wiley \& Sons.

\bibitem{Molc05}
Molchanov, I.S. (2005).
{\em Theory of random sets\/}.
Springer.

\bibitem{MollWaag04}
M\o ller, J., Waagepetersen, R.P. (2003).
{\em Statistical inference and simulation for spatial point processes\/}.
CRC Press.

\bibitem{Moll98}
M\o ller, J., Syversveen, A.R., Waagepetersen, R.P. (1998).
Log Gaussian Cox processes. 
{\em Scandinavian Journal of Statistics\/} 25:451--482.

\bibitem{Nych}
Nychka, D., Furrer, R., Paige, J., Sain, S. (2015).
Fields: Tools for spatial data.
University Corporation for Atmospheric Research, Boulder, Colorado.

\bibitem{Ripl77}
Ripley, B.D. (1977).
Modelling spatial patterns (with discussion).
{\em Journal of the Royal Statistical Society Series B\/} 39:172--212.

\bibitem{Ripl88}
Ripley, B.D. (1988).
{\em Statistical inference for spatial processes\/}.
Cambridge University Press.

\bibitem{mpplib}
Steenbeek, A.G., Lieshout, M.N.M.~van and Stoica, R.S.\
with contributions from P. Gregori, K.K.\ Berthelsen and A.A.\ Iftimi. (2016).
MPPBLIB, a {\tt C++} library for marked point processes. CWI.

\bibitem{StoyOhse82}
Stoyan, D., Ohser, J. (1982).
Correlations between planar random structures with an
ecological application.
{\em Biometrical Journal\/} 24:631--647.

\bibitem{Stoy2}
Stoyan, D., Stoyan, H. (2000).
Improving ratio estimators of second order point process characteristics.
{\em Scandinavian Journal of Statistics\/} 27:641--656.

\bibitem{WidoRowl70}
Widom, B., Rowlinson, J.S. (1970).
New model for the study of liquid-vapor phase transitions.
{\em The Journal of Chemical Physics\/} 52:1670--1684.

\bibitem{Zess83}
Zessin, H. (1983).
The method of moments for random measures.
{\em Probability Theory and Related Fields\/} 62:395--409.

\end{thebibliography}
\end{document}